\documentclass[journal]{IEEEtai}

\usepackage{color}
\usepackage{lineno}
\modulolinenumbers[5]
\usepackage{xpatch}
\usepackage{bm}
\usepackage{array}
\usepackage{graphicx}
\usepackage{bigstrut}
\usepackage{amsmath,amssymb,amsthm}
\usepackage{algpseudocode}
\usepackage{algorithmicx}
\usepackage{algorithm}
\usepackage{booktabs}
\usepackage{color}
\usepackage{changepage}
\usepackage{xr}
\usepackage{xr-hyper}
\usepackage{graphicx}
\usepackage{subfigure}
\usepackage{caption}
\usepackage{makecell}
\usepackage{multirow}
\usepackage{float}
\usepackage{soul}
\usepackage[hidelinks]{hyperref}
\usepackage[numbers,sort&compress]{natbib}
\usepackage{bigstrut} 
\usepackage[table]{xcolor} 

 \newtheorem{definition}{Definition}
 \newtheorem{remark}{Remark}
 \newtheorem{example}{Example}

\newtheoremstyle{nonbolditalic}
  {}
  {}
  {\normalfont}
  {}
  {\itshape}
  {.}
  {.5em}
  {\thmname{#1}\thmnumber{\textit{\enskip#2}}\thmnote{#3}}
\theoremstyle{nonbolditalic}
\newtheorem{proposition}{Proposition}

\AtBeginDocument{%
 \abovedisplayskip=5pt plus 4pt minus 2pt
 \abovedisplayshortskip=5pt plus 4pt minus 4pt
 \belowdisplayskip=5pt plus 4pt minus 2pt
 \belowdisplayshortskip=5pt plus 4pt minus 4pt
}

\ifCLASSINFOpdf

\else

\fi

\begin{document}
\captionsetup{font={footnotesize}}
\captionsetup[table]{labelformat=simple, labelsep=newline, textfont=sc, justification=centering}

\title{A Composite Decomposition Method\\ for Large-Scale Global Optimization}

\author{Maojiang Tian,
        Minyang Chen,
        Wei Du,~\IEEEmembership{Member,~IEEE},
        Yang Tang,~\IEEEmembership{Fellow,~IEEE}, 
        \\
        Yaochu Jin,~\IEEEmembership{Fellow,~IEEE},
        and Gary G. Yen,~\IEEEmembership{Fellow,~IEEE}
        
\thanks{This work was supported by National Natural Science Foundation of China (Key Program: 62136003), National Natural Science Foundation of China (62173144, 62373153), Shanghai Rising-Star Program (22QA1402400), National Natural Science Foundation of Shanghai (21ZR1416100) and Fundamental Research Funds for the Central Universities (222202417006).
(\emph{Corresponding author: Wei Du.})}
        
\thanks{M. Tian, W. Du and Y. Tang are with the Key Laboratory of Smart Manufacturing in Energy Chemical Process, Ministry of Education, East China University of Science and Technology, Shanghai 200237, China (e-mail: mjtian0618@gmail.com; duwei0203@gmail.com; yangtang@ecust.edu.cn).}
\thanks{W. Du is also with the Engineering Research Center of Process System Engineering, Ministry of Education, East China University of Science and Technology, Shanghai 200237, China.}
\thanks{M. Chen is with the Department of Computer Science and Engineering, Southern University of Science and Technology, Shenzhen, 518055, China (cmy1223605455@gmail.com).
}
\thanks{Y. Jin is with the School of Engineering, Westlake University, Hangzhou 310030, China (e-mail: jinyaochu@westlake.edu.cn).}
\thanks{G. Yen is with the School of Electrical and Computer Engineering, Oklahoma State University, Stillwater, OK 74048, USA. (e-mail: gyen@okstate.edu).}}

\maketitle

\begin{abstract}
Cooperative co-evolution (CC) algorithms, based on the divide-and-conquer strategy, have emerged as the predominant approach to solving large-scale global optimization (LSGO) problems. 
The efficiency and accuracy of the grouping stage significantly impact the performance of the optimization process. 
While the general separability grouping (GSG) method has overcome the limitation of previous differential grouping (DG) methods by enabling the decomposition of non-additively separable functions, it suffers from high computational complexity. 
To address this challenge, this article proposes a composite separability grouping (CSG) method, seamlessly integrating DG and GSG into a problem decomposition framework to utilize the strengths of both approaches.
CSG introduces a step-by-step decomposition framework that accurately decomposes various problem types using fewer computational resources. 
By sequentially identifying additively, multiplicatively and generally separable variables, CSG progressively groups non-separable variables by recursively considering the interactions between each non-separable variable and the formed non-separable groups. 
Furthermore, to enhance the efficiency and accuracy of CSG, we introduce two innovative methods: a multiplicatively separable variable detection method and a non-separable variable grouping method. 
These two methods are designed to effectively detect multiplicatively separable variables and efficiently group non-separable variables, respectively.
Extensive experimental results demonstrate that CSG achieves more accurate variable grouping with lower computational complexity compared to GSG and state-of-the-art DG series designs.
\end{abstract}

\begin{IEEEImpStatement}
Numerous challenges in the field of artificial intelligence involve a considerable number of parameters or decision variables.
These challenges often manifest as large-scale global optimization problems, such as neural architecture search and robotic path planning. 
Decomposing such problems into lower-dimensional sub-problems, utilizing the concept of divide-and-conquer, proves to be an effective strategy for addressing these challenges.
However, existing decomposition-based methods struggle to balance the efficiency and accuracy of the decomposition process.
The composite decomposition method presented in this paper achieves both accurate and efficient decomposition of large-scale optimization problems.
This method efficiently explores the structure of large-scale global optimization problems, providing a novel perspective for solving such challenges in the field of artificial intelligence.

\end{IEEEImpStatement}


\section{Introduction}

\IEEEPARstart{T}{he} class of optimization problems with hundreds to thousands of decision variables is referred to as large-scale global optimization (LSGO) problems \cite{LSGO1}.
Nowadays, LSGO problems have become prevalent in various domains of scientific research and practical engineering, such as neural network architecture search \cite{DDG}, financial portfolio optimization \cite{EO}, and air traffic flow management \cite{Traffic}, to name a few.
The ``curse of dimensionality'' presents a significant challenge when dealing with LSGO problems using traditional mathematical procedures since too many variables increase the size of the search space and result in a more complex fitness landscape, making the problem more difficult to solve.
Evolutionary algorithms (EAs) \cite{Handbook} are frequently employed to address LSGO problems because they have excellent global optimization capabilities and little dependency on the nature of the problems.

There are two types of EAs for solving LSGO problems: non-decomposition-based approaches and decomposition-based approaches \cite{Meta, LSGO1}.
The non-decomposition-based approaches optimize all decision variables simultaneously, and representative methods include local search (such as SHADE-ILS \cite{SHADE-ILS}), hybrid algorithms (like MOS \cite{MOS}) and enhanced optimizers (such as CSO \cite{CSO}).
The decomposition-based approaches, inspired by the concept of divide-and-conquer, utilize the cooperative co-evolution (CC) framework \cite{CC1, CC4} to decompose decision variables before optimization.
This approach breaks down the high-dimensional problem into several lower-dimensional subproblems, each of which is individually optimized using EAs.
The accuracy and computational resources of the decomposition phase significantly impact the performance of the CC framework \cite{DG}.
Ensuring the accuracy of the decomposition is vital to prevent correlation among the subsets. 
This ensures that the CC framework can optimize each subset independently, without being influenced by the effects of other subsets.
The allocation of fewer computational resources in the decomposition phase enables the utilization of more computational resources in the optimization phase.
In recent years, numerous decomposition-based approaches have been proposed for solving LSGO problems.

In the early studies, due to the lack of methods to detect the interaction between variables, researchers tended to utilize static decomposition methods \cite{CC1, KS} and random decomposition methods \cite{CC4}. 
However, these methods are incapable of identifying separable and non-separable variables in the problem and are only suitable for fully-separable problems.

To increase the accuracy of decomposition, researchers have applied various learning-based decomposition algorithms to explore the interactions between decision variables.
Learning-based approaches can be classified into three types based on their principles of separability detection: finite differences \cite{DG, RDG, EDG}, monotonicity detection \cite{CCVIL, FVIL, IRRG}, and the minimum points shift principle \cite{GSG, SVG}.
\begin{enumerate}
  \item 
        Finite differences: 
        The grouping methods based on the finite difference principle can accurately identify additively separable problems.
        The first grouping method based on the finite difference principle is referred to as differential grouping (DG) \cite{DG}.
        However, DG suffers from three primary drawbacks that significantly impact its grouping performance. 
        These limitations encompass the detection of interactions between variable pairs, which results in excessive consumption of computational resources; the reliance on a fixed threshold and the exclusion of indirect interactions, which lead to diminished accuracy in grouping.

        To address the computational cost of DG, recursive DG (RDG) \cite{RDG} adopts a set-set recursive interaction detection method that reduces the complexity from $\mathcal{O}(n^2)$ to $\mathcal{O}(nlogn)$.
        Efficient RDG (ERDG) \cite{ERDG} makes full use of historical information to identify interactions, which reduces redundant detection in RDG.
        DG2 \cite{DG2} and merged DG (MDG) \cite{MDG} set adaptive thresholds based on the values of the objective function to reduce the impact of rounding errors on the accuracy of the decomposition.
        Efficient adaptive DG (EADG) \cite{EADG} designs a normalized interdependency indicator without requiring a threshold to be specified.     
        Extended DG (XDG) \cite{XDG} and global DG (GDG) \cite{GDG} propose methods to detect indirect interactions, and RDG3 \cite{RDG3} analyzes the overlapping problem caused by indirect interactions and further groups the overlapping problems by setting threshold.
        In recent years, researchers have attempted to achieve more efficient grouping for overlapping problems.
        For example, graph-based deep decomposition (GDD) \cite{GDD} utilizes the minimum vertex separator to achieve deep grouping. 
        Overlapped RDG (ORDG) \cite{ORDG} extends the set-set recursive interaction detection method to overlapping problems and improves the decomposition efficiency.
        Dynamic CC (DCC) \cite{DCC} decomposes the overlapping problem dynamically based on contribution and historical information.

        Dual DG (DDG) \cite{DDG} employs logarithmic transformation to convert the multiplicatively separable function into an additively separable function. 
        Subsequently, it utilizes the finite difference principle to ascertain the multiplicatively separable variables within the original problem.
        
        In summary, the DG series methods have achieved good results in dealing with additive separability problems.
        Research in this field is extensive and comprehensive.
        \vspace{0.2cm}
  \item 
        Monotonicity detection: The monotonicity detection examines the interaction between two variables based on whether the global optimum of a function can be reached by successive line search along the axes \cite{CCVIL}.
        
        CC with variable interdependence learning (CCVIL) \cite{CCVIL} is the first attempt to apply the monotonicity detection principle.
        Fast variable interdependence learning (FVIL) \cite{FVIL} simultaneously detects the interaction between two sets of variables similar to RDG does, significantly reducing computational resource consumption. 
        Incremental recursive ranking grouping (IRRG) \cite{IRRG} creates two sample point rankings to perform monotonicity tests effectively utilizing sampling information.
        
        Although the monotonicity test can decompose various types of separable problems to some extent, it may overlook interactions between non-separable variables and cannot ensure the accuracy of the decomposition.
        Additionally, there is currently no proven theory to demonstrate the accuracy of identifying the variable separability through the monotonicity detection principle.
\vspace{0.2cm}
  \item 
        Minimum points shift principle: This principle is the latest development in the problem decomposition approach based on the definition of separability.
        Unlike the previous two classes of methods, the decomposition approach based on the minimum points shift principle accurately decomposes both additively and non-additively separable problems.
        
        This type of decomposition method consists of two stages: minimum points search and interaction detection \cite{GSG}.
        In the first stage, independent minimum points are identified for all variables.
        In the second stage, the interaction of each variable with the remaining variables is detected by examining the minimum point.
        
        Surrogate-assisted variable grouping (SVG) \cite{SVG} seeks the global optimum of each variable using two-layer polynomial regression (TLPR) surrogate structure and the BFGS method.
        This approach requires a large number of fitness evaluations (FEs) for each variable in the global optimum search, resulting in significant computational resource consumption. 

        GSG introduces the concepts of strong separability and weak separability by studying the fitness landscape of various types of functions on their subspaces.
        This analysis highlights the significance of strong separability in detecting separability and facilitating decomposition-based optimization.
        GSG then presents a series of definitions for interaction and establishes the principle of minimum points shift.
        By utilizing local minimum points instead of global minimum points, GSG efficiently identifies variable separability, resulting in significantly reduced computational requirements compared to SVG.
        
        Although these decomposition methods are theoretically proven to be comprehensive and capable of decomposing all kinds of separable functions, they generally suffer from high computational complexity.

\vspace{0.2cm}
\end{enumerate}

In summary, the DG series methods have effectively showcased their ability to decompose additively separable problems. 
However, a notable drawback of the DG series methods is its limitation to only decompose additively and multiplicatively separable problems, lacking the capability to identify generally separable variables within the problems. 
In contrast, GSG introduces the principle of minimum points shift, enabling accurate detection of all separable variables in the problems.
Nevertheless, GSG's approach involves multiple function evaluations to progressively narrow down the search interval and locate precise minimum value points for each dimensional variable. 
Consequently, this leads to a significantly high computational complexity.

Based on the above discussions, we propose a novel decomposition method called composite separability grouping (CSG) in this article.
CSG utilizes the strengths of DG, which efficiently detects additively separable, and GSG, which identifies all types of separable variables, developing a composite decomposition framework.
In CSG, the first step involves identifying additively separable variables using the DG method and detecting multiplicatively separable variables using a novel method called multiplicatively separable variable detection (MSVD) proposed in this article.
Subsequently, the remaining variables are subjected to a search for minimum points to detect other separable variables.
During this stage, CSG thoroughly examines the interaction between each detected variable and all other variables. 
This one-to-all detection method ensures reduced computational resource consumption.
Finally, the non-separable variables are grouped based on their interactions using a novel method called non-separable variable grouping (NVG).
For problems with multiple types of separable variables in LSGO, CSG achieves a high level of decomposition accuracy while maintaining low computational complexity.
The main contributions of this work are summarized as follows:

\begin{itemize}
    \item We have developed a composite decomposition method that combines the strengths of both DG and GSG methods. 
    This approach allows a step-by-step decomposition of the problem, resulting in accurate identification of all types of separable variables and non-separable variable groups, while significantly reducing the computational resources required.
    
    \item In order to enhance the integration of DG and GSG, we have devised a new method for detecting multiplicatively separable variables, as well as a grouping method for non-separable variables. 
    These novel approaches significantly improve the accuracy and efficiency of CSG.
    
    \item We have introduced a novel benchmark that encompasses multiple types of separable variables. 
    CSG has been evaluated against GSG and state-of-the-art DG methods using this novel benchmark and classical CEC and BNS benchmarks.
    The experimental results demonstrate that CSG achieves more accurate problem decomposition with lower computational complexity.
\end{itemize}

The remaining article proceeds as follows.
Section \ref{section RELATED WORK} introduces the related definitions of separability and methods of problem decomposition.
Section \ref{section proposed} describes the proposed CSG in detail.
Next, the proposed benchmark and experimental results are presented in Section \ref{section Experimental}.
Finally, Section \ref{section Conclusion} provides the conclusion of this article and outlines future research directions.

\section{Background}\label{section RELATED WORK}
In this section, we will summarize the definitions of separability mentioned in the previous papers \cite{DG, DDG, GSG} and provide a detailed description of the related types of separability and separability detection methods.\subsection{Basic Definitions}
Separability serves as the foundation for decomposing challenging LSGO problems. 
For the completeness of the presentation, this section will introduce the relevant definitions.
\begin{definition}\label{defn:Separability}
\emph{(Separability)}\\
A problem $f(\bm{x})$ is separable if :
\begin{eqnarray}
\begin{split}
  \underset { {x_{1},...,x_{n}}}  {\text{arg min}} f({\bm{x}}) =(\underset{\bm{X_{1}}} {\text{arg min}} f({\bm{X_{1}}}),...,\underset{\bm{X_{k}}} {\text{arg min}} f({\bm{X_{k}}})) 
\end{split}
\end{eqnarray}
where $\bm{x}$ is a vector of {n} decision variables $(x_{1},...,x_{n})$, and $\bm{X_{1}}$,...,$\bm{X_{k}}$ are disjoint subcomponents of $\bm{x}$ which are called non-separable variable group. 
$k$ is denoted as the number of subcomponents.
When $k = n$, $f(\bm{x})$ is a fully separable problem.
Otherwise, $f(\bm{x})$ is a partially separable function.
\end{definition}
\begin{remark}
A problem is considered separable when one or more variables can be optimized independently, without considering the influence of other variables. 
In such cases, the other variables are typically fixed using context vectors. 
The variables involved in a separable function are referred to as ``separable variables'' in this study.
\end{remark}
\begin{remark}
There exist various types of separable functions, which we collectively refer to as generally separable functions.
Additively separable, multiplicatively separable, and composite separable functions are among the few notable examples of this class of separable problems.
\end{remark}

\begin{definition}\label{Additive Separability}
\emph{(Additive Separability)}\\
A problem $f(\boldsymbol{x})$ is additively separable if:
\begin{eqnarray}
\begin{split}
f(\boldsymbol{x})=\sum_{i=1}^{k} f_{i}\left(\boldsymbol{X}_{i}\right), k=2, \ldots, n
\end{split}
\end{eqnarray}
where $\bm{x}$ is a vector of $n$ decision variables $(x_{1},...,x_{n})$, and $\bm{X_{1}}$,...,$\bm{X_{k}}$ are non-separable disjoint subcomponents of $\bm{x}$. 
$k$ is denoted as the number of subcomponents.
When k = n, $f(\bm{x})$ is a fully additively separable problem.
\end{definition}
The additively separable problem is the most extensively studied class of separability.
Many benchmark functions are designed based on additively separable functions \cite{CEC2010, CEC2013}.
A series of methods, utilizing the principle of finite difference, has been developed to detect additively separable functions.

\begin{definition}\label{Multiplicative Separability}
\emph{(Multiplicative Separability)}\\
A problem $f(\bm{x})$ is multiplicatively separable if it satisfies the following two conditions:
\begin{enumerate}
  \item $f(\boldsymbol{x})=\prod_{i=1}^{m} f_{i}\left(\boldsymbol{X}_{i}\right) $;
  \item $f_{i}\left(\boldsymbol{X}_{i}\right) \geq 0 \ or \ f_{i}\left(\boldsymbol{X}_{i}\right)<0, i=1,2, \ldots, m$,
\end{enumerate}
where $\bm{x}$ is a vector of $n$ decision variables $(x_{1},...,x_{n})$, and $\bm{X_{1}}$,...,$\bm{X_{k}}$ are non-separable disjoint subcomponents of $\bm{x}$.
$k$ is denoted as the number of subcomponents.
When $k = n$, $f(\bm{x})$ is fully multiplicatively separable problem.
\end{definition}
The concept of multiplicative separability in LSGO was initially proposed by DDG \cite{DDG}.
GSG \cite{GSG} further enhances the theory of multiplicative separability and provides relevant theoretical proof.
In a multiplicatively separable function, the fitness function value of the subsets cannot be 0; otherwise, the function will lose its separability.

\begin{example} \label{MSV}
An example of a separable function is $f(\bm{x})=x_{1}+\left(x_{2}^{2}+1\right)\cdot\left(x_{3}^{2}+1\right),\ \textbf{x}\in \mathbb{R}^3$.
According to the definition of additive separability, it can be decomposed into $\boldsymbol{x}_1=\{x_1\}, \boldsymbol{x}_2=\left\{x_{2}, x_{3}\right\}$ and $f_1(\boldsymbol{x}_1)=x_1, f_2(\boldsymbol{x}_2)=\left(x_{2}^{2}+1\right)\cdot\left(x_{3}^{2}+1\right)$.
It is evident that $f_2(\boldsymbol{x}_2)$ is a multiplicatively separable function and can be further decomposed. 
$\boldsymbol{x}_{2}=\left\{x_{2}, x_{3}\right\}$ can be decomposed into $\boldsymbol{x}_{21}=\{x_2\}$ and $\boldsymbol{x}_{22}=\{x_3\}$, and $f_2(\boldsymbol{x}_2)=\left(x_{2}^{2}+1\right)\cdot\left(x_{3}^{2}+1\right)$ can be decomposed into $f_{21}\left(\boldsymbol{x}_{21}\right)=\left(x_{2}^{2}+1\right)$ and $f_{22}\left(\boldsymbol{x}_{22}\right)=\left(x_{3}^{2}+1\right)$.
Therefore, a separable function may exhibit multiple types of separability. 
The subfunctions of a separable function can contain various separable types and require a step-by-step decomposition process.
\end{example}

\begin{definition}\label{Composite Separability}
\emph{(Composite Separability)}\\
If $f(\bm{x})$ is a composite function with an inner function $g(\bm{x})$ and an outer function $H(\cdot)$, and the composite function satisfies the following two conditions:
\begin{enumerate}
  \item $g(\textbf{x})$ is a separable function;
  \item $H(\cdot)$ is monotonically increasing in its domain,
\end{enumerate}
then $f(\bm{x})$ is considered compositely separable.
We can optimize each compositely separable variable (or group) of $f(\bm{x})$ individually without considering the interference of other variables (or groups).
\end{definition}
\begin{example}
An illustration of a compositely separable function is $f(\boldsymbol{x})=\sqrt{x_1^2+x_2^2}$.
In this function, $x_1$ and $x_2$ can be optimized independently.
The minimum value obtained by optimizing the inner function $g(\bm{x})={x_1^2+x_2^2}$ corresponds to the minimum value of the outer function $H(\cdot)$.
\end{example}

\subsection{Problem Decomposition Methods}\label{Section:Problem Decomposition Methods}
The learning-based decomposition methods exhibit better performance than static decomposition and random decomposition \cite{DG, CCVIL, GSG}.
The learning-based decomposition methods include monotonicity checking, differential grouping, and the minimum points shift principle. 
These last two methods have complete theoretical proof and are closely related to the method proposed in this article. 
Consequently, we will provide a detailed description of them.
\subsubsection{Differential Grouping}
The DG method identifies variable interaction by detecting changes in the fitness function values when perturbing the variables.
Consider a function $f(\boldsymbol{x})$ that is additively separable. DG determines that two variables interact if the following condition is satisfied,
\begin{eqnarray}
\begin{split}
\left.\Delta_{\delta, {x}_{p}}[f](\boldsymbol{x})\right|_{{x}_{p}=a, {x}_{q}=b_{1}} \neq\left.\Delta_{\delta, {x}_{p}}[f](\boldsymbol{x})\right|_{{x}_{p}=a, {x}_{q}=b_{2}}
\end{split}
\end{eqnarray}
where
\begin{eqnarray}
\begin{split}
{\Delta_{\delta, {x}_{p}}[f](\boldsymbol{x})=f\left(\ldots, {x}_{p}+\delta, \ldots\right)-f\left(\ldots, {x}_{p}, \ldots\right)}
\end{split}
\end{eqnarray}

DG checks all variables in pairs to determine if they interact.
If two variables interact, they are placed into the same non-separable variable group.
A variable is considered separable if it does not interact with any other variables in the problem.
The computational complexity of DG is excessively high.

To address the aforementioned problems, RDG is employed, which examines the interaction between decision variables using nonlinearity detection and binary search.
Let $\bm{X_{1}}$ and $\bm{X_{2}}$ represent two disjoint groups of variables that are subsets of $\bm{X}$.
Suppose there exist two unit vectors, $\bm{u}_{1} \in U_{\bm{X_{1}}}$ and $\bm{u}_{2} \in U_{\bm{X_{2}}}$, and two real numbers $l_{1}, l_{2}>0$, along with a candidate solution $\bm{x}^{*}$ in the decision space, such that
\begin{eqnarray}
\begin{split}
\begin{aligned}
f\left(\bm{x}^{*}+l_{1} \bm{u}_{1}+l_{2} \bm{u}_{2}\right)- & f\left(\bm{x}^{*}{+}l_{2} \bm{u}_{2}\right)  \\
{\neq} & f\left(\bm{x}^{*}+l_{1} \bm{u}_{1}\right)-f\left(\bm{x}^{*}\right)
\end{aligned}
\end{split}
\end{eqnarray}
$\bm{X_{1}}$ interacts with $\bm{X_{2}}$, and there exist variables in groups $\bm{X_{1}}$ and $\bm{X_{2}}$ that belong to the same group of non-separable variable group.
It can be noticed that RDG only requires a single one-to-all interaction check to determine whether a particular variable $x$ is fully separable.

DDG establishes a mathematical relationship between additively and multiplicatively separable functions and deduces how they can be converted into each other.
For the multiplicatively separable function $f(\bm{x})$ mentioned in Definition \ref{Multiplicative Separability}, if the minimum value of the function is greater than 0, it can be transformed into an additively separable function using the logarithmic function. The relevant procedure is demonstrated below:
\begin{eqnarray}
\begin{split}
\begin{aligned}
g(\bm{x}) & =\ln f(\bm{x}) \\
& =\ln \prod_{i=1}^{k} f_{i}\left(\bm{X}_{i}\right) \\
& =\sum_{i=1}^{k} \ln f_{i}\left(\bm{X}_{i}\right), 1<k \leq n
\end{aligned}
\end{split}
\end{eqnarray}
According to Definition 2, $g(\bm{x})$ is an additively separable function. Therefore, we can decompose it using DG series methods. 
Consequently, DDG determines that the two variables in the function $f(\bm{x})$ are multiplicatively separable if the following condition holds:
\begin{equation} \label{equ8}
\begin{aligned}
&\left. \Delta_{\delta, x_{p}} \ln \left\{ [f](\bm{x}) \right\} \right|_{x_{p}=a, x_{q}=b_{1}} \\
&\quad \quad \quad \quad \quad \quad \quad \quad \neq \left. \Delta_{\delta, x_{p}} \ln \left\{ [f](\bm{x}) \right\} \right|_{x_{p}=a, x_{q}=b_{2}}
\end{aligned}
\end{equation}
where 
\begin{eqnarray} \label{equ9}
\begin{aligned}
&\Delta_{\delta, x_{p}} \ln \{ [f](\bm{x}) \} \\
&\quad \quad \quad = \ln [f\left(\ldots, x_{p}+\delta, \ldots\right)] - \ln [f\left(\ldots, x_{p}, \ldots\right)]
\end{aligned}
\end{eqnarray}
The time complexity of DDG is $\mathcal{O}(n^2)$, which is the same as DG, where $n$ represents the number of variables or the size of the problem.

However, DDG is limited to decomposing multiplicatively separable functions and may not be able to identify the multiplicatively separable variables in functions such as the one given in Example \ref{MSV}.
This limitation arises because a logarithmic transformation cannot convert non-multiplicatively separable functions into additively separable ones. 
Consequently, DDG may not effectively decompose certain functions when their subfunctions contain multiplicatively separable variables.

\subsubsection{Minimum Points Shift Principle}
According to the definition of separability, the condition for two variables, $x_{i}$ and $x_{j}$, to be separable is whether their optimal values are independent.
This condition can be detected using the following method: if the target variable, $x_{i}$, and the undetected variable, $x_{j}$, are separable, then the global optimal value of $x_{i}$ remains unchanged even after perturbing $x_{j}$.
To avoid repeatedly detecting the optimal value of $x_{i}$, the above detection method can be transformed into detecting whether the previous optimal value of $x_{i}$ remains unchanged after perturbing the $x_{j}$.

Let $f(\bm{x})$ be a $D$-dimensional function, and let $x_{i}^{*}$ represent the global minimum of the variable $x_{i}$.
$\delta$ represents the smallest perturbation, and $\emph{\textbf{cv}}$ denotes the context vector in the decision space $\mathbb{R}^{D}$.
The target variable $x_{i}$ is separable from other variables $\bm{X_{i}}$ if the following condition is met:
\begin{eqnarray}
\begin{split}
\begin{aligned}
\min \left\{f\left(\emph{\textbf{cv}} \mid \leftarrow x_{i}^{*}-\delta, \bm{X_{i}}\right), f\left(\emph{\textbf{cv}} \mid \leftarrow x_{i}^{*}+\delta, \bm{X_{i}}\right)\right\} \\ >{f\left(\emph{\textbf{cv}} \mid \leftarrow x_{i}^{*}, \bm{X_{i}}\right)}
\end{aligned}
\end{split}
\end{eqnarray}
However, searching for the global optimum value of a variable in the entire decision space is extremely challenging and can consume a significant amount of computational resources.

GSG \cite{GSG} demonstrates that the local minimum point can serve as a substitute for the global minimum point in separability detection for strongly separable functions, thus reducing the search difficulty compared to SVG \cite{SVG}. 

However, GSG also incurs a high computational cost, requiring an average of over $5 \times 10^4$ FEs to decompose 1000-D BNS and CEC problems \cite{GSG}, which is several times more than state-of-the-art DG series methods. 
According to \cite{GSG}, approximately 70\% of the total computational resources in BNS \cite{GSG} and CEC \cite{CEC2010, CEC2013} problems are consumed by searching for the minimum points. 
Upon analyzing the decomposition principle of GSG, it becomes evident that the same approach is applied to different types of variable, including searching for the minimum point of each variable and conducting one-to-all interaction detection for each variable. 
Consequently, GSG may expend significant computational resources to identify variables that are relatively easy to detect.
Although GSG, based on the minimum points shift principle, achieves high decomposition accuracy, finding ways to reduce its computational resource consumption has become a pressing challenge.

\section{The Proposed CSG}\label{section proposed}
To address the challenge of excessive computational resource consumption associated with the GSG method, we propose a novel approach called CSG. 
This section provides a detailed overview of the specific components of CSG.

Firstly, we introduce an algorithmic framework for CSG that efficiently decomposes the currently proposed separable variable types and groups non-separable variables in LSGO problems. 
Next, we present a more generalized decomposition method specifically designed to identify multiplicatively separable variables. 
This method incorporates historic information obtained during the detection of additively separable variables, enabling accurate and efficient identification of multiplicatively separable variables.
Furthermore, within the framework, we propose a recursive detection method based on non-separable variable groups.
This method efficiently groups non-separable variables.
Lastly, we analyze the theoretical computational complexity of CSG, providing insights into the computational requirements associated with implementing the proposed method.

\subsection{The Framework of CSG}
CSG consists of four stages: additively separable variable identification, multiplicatively separable variable identification, generally separable variable identification, and non-separable variable grouping. 
In this paper, when we mention generally separable variables, we are specifically referring to non-additive and non-multiplicative separable variables.
The aim of the first three stages is to identify all separable variables by checking the separability of variables dimension by dimension.
The last stage focuses on grouping the remaining non-separable variables.

The framework of CSG is presented in Algorithm \ref{algorithm_CSG}.
After initialization (Lines 1-2), the corresponding fitness function values are recorded when $\emph{\textbf{x}}$ takes the upper bounds $\emph{\textbf{x}}_{u,u}$ and lower bounds  $\emph{\textbf{x}}_{l,l}$ (Line 3). 
These values are stored for later detection to avoid repeated evaluations.
CSG examines whether each variable is additively separable using the finite difference principle (Lines 5-8). 
If there is no interaction between the current variable $x_i$ and the remaining variables $\{x_1,...,x_{i-1},x_{i+1},...,x_D  \}$, the variable $x_i$ is considered additively separable and is placed into $S_1$ (Line 10). 
Otherwise, the current variable $x_i$ continues to be examined for multiplicatively separable detection (Line 12).

\begin{algorithm}[!h]
\small
\caption{CSG}
\label{algorithm_CSG}
\begin{algorithmic}[1]
\Require{\emph{f}, $D$, $S$ (all variables), \emph{\textbf{ub}}, \emph{\textbf{lb}}, $\epsilon$ (precision of golden section search), $\alpha$ (scale factor of initial detection step)}
\Ensure{$S_1$ (additively separable variables), $S_2$ (multiplicatively separable variables), $S_3$ (non-additively and non-multiplicatively separable variables), $N$ (groups of non-separable variables)}
\State $S_1,S_2,S_3 \gets [\ ], N \gets \{\}$
\State $\textbf{\emph{C}}_{arc} \gets {\frac{\emph{\textbf{ub}}+\emph{\textbf{lb}}}{2}_{D\times D}}$, $\emph{\textbf{cv}} \gets \frac{\emph{\textbf{ub}}+\emph{\textbf{lb}}}{2}$
\State $\emph{\textbf{x}}_{l,l} \gets \emph{\textbf{lb}}$, $\emph{f}_{l,l} \gets \emph{f}(\emph{\textbf{x}}_{l,l})$, $\emph{\textbf{x}}_{u,u} \gets \emph{\textbf{ub}}$, $\emph{{f}}_{u,u} \gets \emph{{f}}(\emph{\textbf{x}}_{u,u})$
\For   {$i\gets$ $D$ to $1$} 
\State  $A=\{1, \ldots, i-1, i+1, \ldots, D\}$
\State  $\emph{\textbf{x}}_{l,u} \gets \emph{\textbf{x}}_{l,l}$, $\emph{\textbf{x}}_{l,u}(A) \gets \emph{\textbf{ub}}(A)$, $\emph{{f}}_{l,u} \gets \emph{{f}}(\emph{\textbf{x}}_{l,u})$
\State  $\emph{\textbf{x}}_{u,l} \gets \emph{\textbf{x}}_{l,l}$, $\emph{\textbf{x}}_{u,l}(i) \gets \emph{\textbf{ub}}(i)$, $\emph{{f}}_{u,l} \gets \emph{{f}}(\emph{\textbf{x}}_{u,l})$
\State $\Delta_{1} \gets \left( \emph{{f}}_{u,l} - \emph{{f}}_{l,l}\right )$, $\Delta_{2} \gets \left( \emph{{f}}_{u,u} - \emph{{f}}_{l,u}\right )$, ${\beta}_{1} \gets \lvert\Delta_{1}-\Delta_{2}\rvert$
\If {$\beta_{1}$ $<$ $\epsilon_{1}$}
\State $S_{1} \gets S_{1} \bigcup \emph{{x}}_{i}$
\Else \State ${\beta}_{2} \gets$ MSVD ($i$, $\emph{\textbf{x}}_{l,l}$, $\emph{\textbf{x}}_{l,u}$, $\emph{\textbf{x}}_{u,l}$, $\emph{\textbf{x}}_{u,u}$, $\emph{{f}}_{l,l}$, $\emph{{f}}_{l,u}$, $\emph{{f}}_{u,l}$, $\emph{{f}}_{u,u}$)
\If {$\beta_{2}$ $<$ $\epsilon_{2}$}
\State $S_{2} \gets S_{2} \bigcup \emph{{x}}_{i}$
\Else \State $\emph{\textbf{x}} \gets [\emph{\textbf{cv}}(1: i-1)), \emph{{x}}_i, \emph{\textbf{cv}}(i+1: D))]$
\State $\emph{\textbf{cv}}(i) \gets$ GSS ($f(\emph{\textbf{x}})$, $\emph{\textbf{ub}}(i)$, $\emph{\textbf{lb}}(i)$, $\epsilon$) /*Search for 
\Statex \quad \; \quad \; \; \; independent minimum point of variable $x_i$*/
\State $\textbf{\emph{C}}_{arc}(i,:) \gets \emph{\textbf{cv}} $
\EndIf
\EndIf
\EndFor
\State $S \gets S / \{S_{1} \bigcup S_{2}\}$
\State ${S}_{3} \gets$ GSVD ($S$, $\textbf{\emph{C}}_{arc}$, $\delta$, $\emph{\textbf{ub}}$) {/*Detect the non-additively and
non-multiplicatively separable variables */}
\State $V \gets S / \{S_{1} \bigcup S_{2} \bigcup S_{3}\}$
\If {$V$ $\neq$ $\emptyset$}
\State $N \gets$ NVG ($V$, $\textbf{\emph{C}}_{arc}$, $\delta$, $\emph{\textbf{ub}}$){/*Form the non-separable
\Statex \quad \; variable groups*/}
\EndIf
\State\Return $S_1$, $S_2$, $S_3$, $N$
\end{algorithmic}
\end{algorithm}

We propose an enhanced detection method called multiplicatively separable variable detection (MSVD), which can identify a broader range of multiplicatively separable variables in LSGO problems.
MSVD utilizes the historical information (Lines 6-7) generated by each variable during the additive separability detection stage.
If a variable is determined to be multiplicatively separable, it is placed into $S_2$ (Line 14).
Otherwise, GSS is applied to search for the independent minimum point of the variable in a reverse order. 
The minimum value point obtained by GSS is recorded in $\emph{\textbf{cv}}$ (Line 17).
$\emph{\textbf{cv}}$ continuously updates with the new minimum points, and the archive matrix $\textbf{\emph{C}}_{arc}$ is used to record the $\emph{\textbf{cv}}$ values generated by the GSS process (Line 18).
The values of $\emph{\textbf{cv}}$ and $\textbf{\emph{C}}_{arc}$ for the previously identified separable variables in $S_1$ and $S_2$ are set to $(\bm{lb}+\bm{ub})/{2}$, as these separable variables do not affect the other variables.

\vspace{-0em}

After Lines 4-21 in Algorithm \ref{algorithm_CSG}, CSG obtains the additively separable variables, the multiplicatively separable variables, and the minimum points of the remaining variable.
CSG employs the generally separable variables detection (GSVD, Algorithm \ref{algorithm_GSVD}) method to identify all generally separable variables. 
In GSVD (Line 23), the interaction between each variable $x_i$ and the other variables $\emph{\textbf{X}}_i$ in $S$ is examined to determine whether it is a generally separable variable.
The $\emph{\textbf{cv}}$ corresponding to the variable being detected is loaded from $\textbf{\emph{C}}_{arc}$ (Line 2), and the values of the other variables in $\emph{\textbf{cv}}$ (excluding the detected variable) are fixed as constants (Lines 3-4).
It is necessary to ensure that the fitness function values of vector $\emph{\textbf{x}}$ are different from the fitness function values of vectors $\emph{\textbf{x}}_l$ and $\emph{\textbf{x}}_r$ by applying incremental and decremental perturbations to the detecting variable $\emph{\textbf{x}}$ (Lines 6-13).
The fitness function values of the three vectors $\emph{\textbf{x}}$, $\emph{\textbf{x}}_l$, and $\emph{\textbf{x}}_r$ are compared, and the separability of variable $x_i$ is determined (Lines 14-16).
If the fitness function value of $\emph{\textbf{x}}$ is not the smallest among the three, then the variable $x_i$ is not a generally separable variable.
Otherwise, the variable is considered a generally separable variable and is placed into $S_3$ (Line 15).

After detecting all separable variables, CSG utilizes the non-separable variable grouping (NVG) method to group the interacting variables together (Lines 25-27).

We introduce the CSG grouping process further with a brief example, as depicted in Fig. \ref{pic4}.
In the example, $x_1$ is an additively separable variable, $x_2$ and $x_3$ are multiplicatively separable variables, $x_4$ and $x_5$ are composite separable variables, and $x_6$ and $x_7$ together constitute a group of non-separable variables.
CSG decomposes the problem sequentially and ultimately identifies three separable variable groups, along with one group of non-separable variables.

\begin{figure*}[htbp]
\centerline{\includegraphics[width=0.75\linewidth]{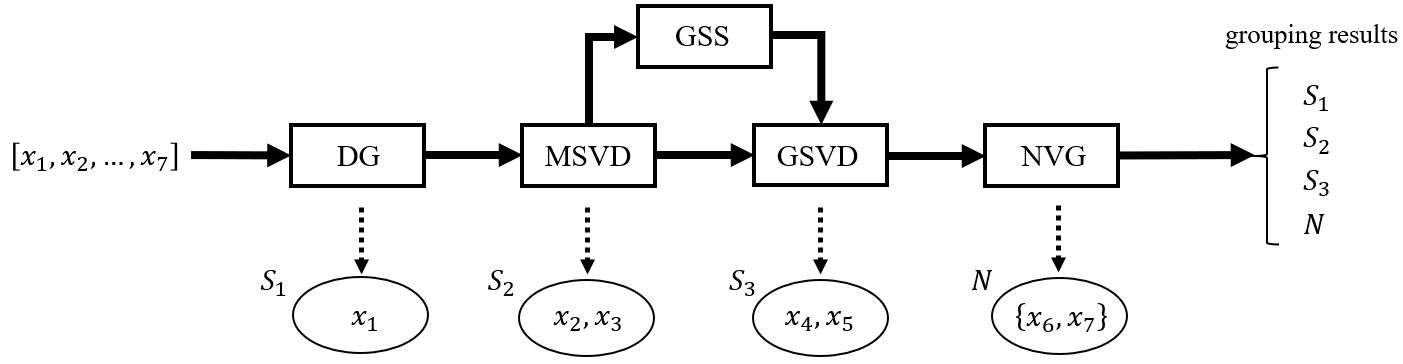}}
\caption{ An example to illustrate the grouping process of CSG with a problem $f(\boldsymbol{x})=x_{1}+x_{2} \cdot x_{3}+\sqrt{x_{4}+x_{5}}+\left(x_{6}-x_{7}-1\right)^{2}$}
\label{pic4}
\vspace{-0.5cm}
\end{figure*}

\begin{algorithm}[!h]
\small
\caption{GSVD}
\label{algorithm_GSVD}
\begin{algorithmic}[1]
\Require{$S$ {(all input variables)}, $\textbf{\emph{C}}_{arc}$ {(archival matrix)}, $\delta$ {(perturbation value)}, $\emph{\textbf{ub}}$}
\Ensure{ $S_3$ {(non-additively and non-multiplicatively separable variables)}}
\For   {each variable $x_i$ in \emph{\textbf{S}}} 
\State $\emph{\textbf{x}} \gets \textbf{\emph{C}}_{arc}(x_i,:)$
\State $S \gets S / x_i $
\State $\emph{\textbf{x}}(S) \gets \emph{\textbf{ub}}(S)$
\State $\emph{\textbf{x}}_l \gets \emph{\textbf{x}}, \emph{\textbf{x}}_r \gets \emph{\textbf{x}}$
\While $ \enspace f(\emph{\textbf{x}}_l)-f(\emph{\textbf{x}}) =0 \vee f(\emph{\textbf{x}}_r)-f(\emph{\textbf{x}})=0$ 
\State $\emph{\textbf{x}}_l (x_i) \gets \emph{\textbf{x}} (x_i) - \delta$
\State $\emph{\textbf{x}}_r (x_i) \gets \emph{\textbf{x}} (x_i) + \delta$
\State $\delta \gets \delta*10 $
\If {$\emph{\textbf{x}}_l (x_i)$ or $\emph{\textbf{x}}_r (x_i)$ is out of boundaries}
\State break
\EndIf
\EndWhile
\If {$f(\emph{\textbf{x}}_l)-f(\emph{\textbf{x}})> 0 \wedge f(\emph{\textbf{x}}_r)-f(\emph{\textbf{x}}) > 0$}
\State $S_{3} \gets S_{3} \bigcup \emph{{x}}_{i}$
\EndIf
\EndFor
\State\Return $S_3$
\end{algorithmic}
\end{algorithm}

\subsection{Multiplicatively Separable Variable Detection}

The proposed method, multiplicatively separable variable detection (MSVD), begins by extracting the subfunctions that contain the detected variable $x_i$ in the additively separable problem.
Next, the extracted functions are transformed by a logarithmic operation.
The difference between the detected variable $x_i$ and the remaining variables in the transformed function is calculated using the finite difference principle.
Finally, the calculated difference is compared to a threshold to determine if the variable $x_i$ satisfies the multiplicative separability.

\begin{proposition} \label{Pro1}
If $f(\boldsymbol{x})$ is an additively separable function consisting of a multiplicatively separable subfunction $f_{i}(\boldsymbol{X}_i)$ with a multiplicatively separable variable $x_i$.
Let $\boldsymbol{x}_1=({x}_1,...,{x}_i,...,{x}_n)$ and $\boldsymbol{x}_2=({x}_1,...,{x}_i/2,...,{x}_n)$,
then $F(\boldsymbol{x})=f(\boldsymbol{x}_1)-f(\boldsymbol{x}_2)$ is a multiplicatively separable function.
\end{proposition}

\begin{proof}
Since $f(\boldsymbol{x})$ is an additively separable function, 
\begin{eqnarray}\label{equ consecutive opt}
\begin{split}
f(\textbf{\emph{x}}_1) &=f_1(\emph{\textbf{X}}_1)+...+f_i(\emph{\textbf{X}}_i)+...f_m(\emph{\textbf{X}}_m)\\&= f_i(\emph{\textbf{X}}_i)+ \sum_{j=1, j \neq i}^{m}f_j( \emph{\textbf{X}}_j)\\&= f_i(\emph{\textbf{X}}_i) + \sum
\end{split}
\end{eqnarray}
Here we abbreviate $\sum_{j=1, j \neq i}^{m}f_j( \emph{\textbf{X}}_j)$ as $\sum$.
$f_{i}(\boldsymbol{X}_i)$ is a multiplicatively separable function with a multiplicatively separable variable $x_i$.
\begin{eqnarray}\label{equ consecutive opt}
\begin{split}
f(\textbf{\emph{x}}_1) &=x_i \cdot f_{i}^{'}(\emph{\textbf{X}}_i^{’})  + \sum 
\end{split}
\end{eqnarray}
Similarly, for vector $\boldsymbol{x}_2$
\begin{eqnarray}\label{equ consecutive opt}
\begin{split}
f(\textbf{\emph{x}}_2) &=\frac{1}{2} \cdot x_i \cdot f_{i}^{'}(\emph{\textbf{X}}_i^{'})  + \sum
\end{split}
\end{eqnarray}
Then
\begin{eqnarray}\label{equ consecutive opt}
\begin{split}
F(\boldsymbol{x})=f(\boldsymbol{x}_1)-f(\boldsymbol{x}_2)=\frac{1}{2} \cdot x_i \cdot f_{i}^{'}(\emph{\textbf{X}}_i) 
\end{split}
\end{eqnarray}
\end{proof}
It can be observed that $F(\boldsymbol{x})$ is a multiplicatively separable function, and we can utilize Eq. (\ref{equ8}) and Eq. (\ref{equ9}) to detect multiplicatively separable variable $x_i$.

The proposed method, MSVD, is presented in detail in Algorithm \ref{algorithm_MSVD}.
In Lines 1-6, MSVD transforms the four detected functions in Eqs. (\ref{equ8}) and (\ref{equ9}) according to Proposition \ref{Pro1} and generates four new functions for multiplicatively separable variable detection.
Based on the vector $\boldsymbol{x}_1$$=({x}_1,...,{x}_i,...,{x}_n)$ in the previous stage, we construct a new vector $\boldsymbol{x}_2$$=({x}_1,...,{x}_i/2,...,{x}_n)$ by multiplying the detected variable in the vector by a factor (set to 0.5 in this article) (Lines 1-4).
MSVD extracts four subfunctions that contain the detected variables $x_i$ in the additively separable problem (Lines 5-6).
MSVD makes full use of the function evaluation value obtained during the additively separable detection stage, effectively preventing redundant detection and conserving computational resources.
The four new extracted functions are transformed by a logarithmic operation and the difference value is obtained according Eqs. (\ref{equ8}) and (\ref{equ9}) (Lines 7-9).
The calculated difference value is compared to a threshold to determine whether $x_i$ is a multiplicatively separable variable.

\begin{algorithm}[!h]
\small
\caption{MSVD}
\label{algorithm_MSVD}
\begin{algorithmic}[1]
\Require{$i$, $\emph{\textbf{x}}_{l,l}$, $\emph{\textbf{x}}_{l,u}$, $\emph{\textbf{x}}_{u,l}$, $\emph{\textbf{x}}_{u,u}$, $\emph{{f}}_{l,l}$, $\emph{{f}}_{l,u}$, $\emph{{f}}_{u,l}$, $\emph{{f}}_{u,u}$}
\Ensure{$\beta_{2}$}
\State $\emph{\textbf{x}}_{l,l}^{'} \gets \emph{\textbf{x}}_{l,l}$, $\emph{\textbf{x}}_{l,l}^{'} (i) \gets \frac{1}{2} \emph{\textbf{lb}}(i)$, $\emph{{f}}_{l,l}^{'} \gets \emph{{f}}(\emph{\textbf{x}}_{l,l}^{'})$
\State $\emph{\textbf{x}}_{u,l}^{'} \gets \emph{\textbf{x}}_{u,l}$, $\emph{\textbf{x}}_{u,l}^{'} (i) \gets \frac{1}{2} \emph{\textbf{ub}}(i)$, $\emph{{f}}_{l,l}^{'} \gets \emph{{f}}(\emph{\textbf{x}}_{u,l}^{'})$
\State $\emph{\textbf{x}}_{l,u}^{'} \gets \emph{\textbf{x}}_{l,u}$, $\emph{\textbf{x}}_{l,u}^{'} (i) \gets \frac{1}{2} \emph{\textbf{lb}}(i)$, $\emph{{f}}_{l,u}^{'} \gets \emph{{f}}(\emph{\textbf{x}}_{l,u}^{'})$
\State $\emph{\textbf{x}}_{u,u}^{'} \gets \emph{\textbf{x}}_{u,u}$, $\emph{\textbf{x}}_{u,u}^{'} (i) \gets \frac{1}{2} \emph{\textbf{ub}}(i)$, $\emph{{f}}_{u,u}^{'} \gets \emph{{f}}(\emph{\textbf{x}}_{u,u}^{'})$
\State $\emph{{F}}_{l,l} = \emph{{f}}_{l,l} - \emph{{f}}_{l,l}^{'}$, $\emph{{F}}_{u,l} = \emph{{f}}_{u,l} - \emph{{f}}_{u,l}^{'}$
\State $\emph{{F}}_{l,u} = \emph{{f}}_{l,u} - \emph{{f}}_{l,u}^{'}$, $\emph{{F}}_{u,u} = \emph{{f}}_{u,u} - \emph{{f}}_{u,u}^{'}$
\State $\Delta_{1}^{'} = \ln \lvert \emph{{F}}_{l,l} \rvert $ $-$ $\ln \lvert \emph{{F}}_{u,l} \rvert $
\State $\Delta_{2}^{'} = \ln \lvert \emph{{F}}_{l,u} \rvert $ $-$ $\ln \lvert \emph{{F}}_{u,u} \rvert$
\State ${\beta}_{2}= \lvert \Delta_{1}^{'}-\Delta_{2}^{'} \rvert $
\State\Return $\beta_{2}$
\end{algorithmic}
\end{algorithm}

\subsection{Non-separable Group Detection}
CSG identifies all separable variables in the first three stages.
To efficiently group the non-separable variables into multiple mutually exclusive subgroups, CSG employs a binary detection method based on the non-separable variable groups in the final stage.
We refer to this method as non-separable variable grouping (NVG).


In the last stage of CSG, multiple groups of non-separable variables exist simultaneously.
The main idea of NVG is to extract the detected variable $x_i$ dimension by dimension and then use a recursive group detection (RGD) method to detect the interactions between the variable $x_i$ and the formed non-separable variable groups.

The pseudocode of NVG is presented in Algorithm \ref{algorithm_NVG}.
CSG begins by initializing the first non-separable variable group, denoted as $N_1$, with the first variable in $V$ (Line 1). 
Then, the algorithm proceeds to examine the second variable in $V$.
The RGD method is employed to detect the interaction between the detected variable $x_i$ and the existing variable groups $N$ (Line 4).
The non-separable groups detected by RGD, which interact with the detected variable $x_i$, are denoted as $N^*$.
If there is no interaction between $x_i$ and the existing variable groups $N$, the variable $x_i$ forms a new non-separable variable group and is added to the set $N$ (Line 6).
In the case where the variable $x_i$ interacts with only one non-separable variable group $N_i$, the variable is added to that specific group $N_i$ (Lines 8-9).
Otherwise, it indicates that there are indirect interactions between the non-separable variable groups in $N^*$ when the number of groups in $N^*$ is greater than one. 
To address this, all groups in $N^*$ are merged into a single non-separable group, and subsequently, the variable $x_i$ is added to this merged group (Lines 11-12).

\begin{algorithm}[!h]
\small
\caption{NVG}
\label{algorithm_NVG}
\begin{algorithmic}[1]
\Require{$V$ {(all input variables)}, $\textbf{\emph{C}}_{arc}$ {(archival matrix)}, $\delta$ {(perturbation value)}, $\emph{\textbf{ub}}$}
\Ensure{$N$ (set of non-separable variable groups $N_i$ )}
\State $N_1 \gets x_r$ {/*Randomly choose a variable $x_r$ to form the first non-separable variable group*/}
\State $V \gets V / x_r$ 
\For   {each variable $x_i$ in $V$}
\State $N^*$ $\gets $ RGD ($x_i$, $V$, $\textbf{\emph{C}}_{arc}$, $\delta$, $ub$) /*find the groups that
\Statex \quad \  interact with $x_i$ in the existing non-separable 
\Statex \quad \  variable groups*/
\If {$|N^*|==0$}
\State $N \gets N \cup x_i$ {/*Utilize $x_i$ to form a new non-separable 
\Statex \qquad \quad variable group*/}
\If{$|N^*|==1$}
\State $N \gets N / N_i$
\State $N \gets N\cup \{N_i \cup x_i\}$ 
\Else \State $N \gets N / \{N_i,...,N_j\}$
\State $N \gets N\cup \{N_i \cup N_j \cup ...\cup N_k \cup x_i\}$
\EndIf
\EndIf
\EndFor
\State\Return $N$
\end{algorithmic}
\end{algorithm}

\vspace{-0cm}

The pseudocode of RGD is presented in Algorithm \ref{algorithm_RGD}.
In Lines 1-11, RGD is similar to GSVD.
RGD first extracts the corresponding context vector $\emph{\textbf{x}}$ of the detected variable $x_i$ from $\textbf{\emph{C}}_{arc}$ (Line 1).
It fixes the values in $\emph{\textbf{x}}$ that exist in $N$ to the upper bound (Line 2).
Then, incremental perturbation and decremental perturbation are applied to $\emph{\textbf{x}}$ to obtain new vectors $\emph{\textbf{x}}_l$ and $\emph{\textbf{x}}_r$ until the fitness function values of both $\emph{\textbf{x}}_l$ and $\emph{\textbf{x}}_r$ differ from the fitness function value of the original vector $\emph{\textbf{x}}$ (Lines 4-11).
While GSVD only performs the variable-to-set interaction detection once to determine whether a variable is generally separable, NVG applies a recursive procedure to find all non-separable variable groups that interact with the variable $x_i$ (Lines 12-23).
Based on the minimum points shift principle, if the fitness value of $\emph{\textbf{x}}$ is the smallest among the three fitness values, it is proven that there is no interaction between $x_i$ and the non-separable variable groups in the current $N$. 
If $x_i$ interacts with $N$ and the number of groups in $N$ is 1, the group is output directly (Lines 13-14).
Otherwise, all groups in set $N$ are split in half into two sets (Line 16), and RGD is recursively called until all groups of non-separable variables are found (Lines 17-19).

NVG is proposed for cases when only non-separable variables exist in the last stage of CSG.
Unlike previous approaches that identify multiple non-separable variable groups dimension by dimension and then merge indirectly interacting variable groups based on shared variables, NVG combines the merging of indirectly interacting variables with the grouping process.
Additionally, NVG has the advantage of computational complexity when dealing with non-separable problems characterized by a small number of subgroups and the absence of indirect interactions.

\begin{algorithm}[!h]
\small
\caption{RGD}
\label{algorithm_RGD}
\begin{algorithmic}[1]
\Require{$x_i$, $V$ {(all input variables)}, $\textbf{\emph{C}}_{arc}$ {(archival matrix)}, $\delta$ {(perturbation value)}, $\emph{\textbf{ub}}$}
\Ensure{$N^*$ (non-separable groups that interact with $x_i$ in $N$)}
\State $\emph{\textbf{x}} \gets \textbf{\emph{C}}_{arc}(x_i,:)$
\State $\emph{\textbf{x}}(N) \gets \emph{\textbf{ub}}(N)$
\State $\emph{\textbf{x}}_l \gets \emph{\textbf{x}}, \emph{\textbf{x}}_r \gets \emph{\textbf{x}}$
\While {$ \enspace f(\emph{\textbf{x}}_l)-f(\emph{\textbf{x}}) =0 \vee f(\emph{\textbf{x}}_r)-f(\emph{\textbf{x}})=0$} 
\State $\emph{\textbf{x}}_l (x_i) \gets \emph{\textbf{x}} (x_i) - \delta$
\State $\emph{\textbf{x}}_r (x_i) \gets \emph{\textbf{x}} (x_i) + \delta$
\State $\delta \gets \delta*10 $
\If {$\emph{\textbf{x}}_l (x_i)$ or $\emph{\textbf{x}}_r (x_i)$ is out of boundaries}
\State break
\EndIf
\EndWhile
\If {$f(\emph{\textbf{x}}_l)-f(\emph{\textbf{x}})< 0 \vee f(\emph{\textbf{x}}_r)-f(\emph{\textbf{x}}) < 0$}
\If {$|N| = 1$}
\State $N^* \gets N$
\Else \State split $N$ into two subsets $N_1^*$ and $N_2^*$ of equal size
\State $N_1^* \gets $ RGD ($x_i$, $ N_{1}$, $\textbf{\emph{C}}_{arc}$, $\emph{\textbf{ub}}$, $\delta$)
\State $N_2^* \gets $ RGD ($x_i$, $N_{2}$, $\textbf{\emph{C}}_{arc}$, $\emph{\textbf{ub}}$, $\delta$) {/*Call RGD recursively
\Statex \qquad \quad until all variables that interact with $x_i$ are found*/}
\State $N^* \gets N_1^* \cup N_2^*$
\EndIf
\Else \State $N^* \gets \emptyset$
\EndIf
\State\Return $N^*$
\end{algorithmic}
\end{algorithm}


\subsection{Computational Resource Consumption Analysis}

In this subsection, we will analyze the computational resource consumption of CSG compared to GSG in the context of a $n$-dimensional problem with $p$ additively separable variables and $q$ multiplicatively separable variables.

First, we analyze the omputational resource consumption of GSG.
GSG involves two main stages: the minimum points search stage and the interaction detection stage.
GSG identifies the minimum value point for each variable and employs a recursive interaction detection process to examine the inter-action of the detected variable with other variable.
In a study conducted using the CEC and BNS benchmarks in \cite{GSG}, the minimum points search stage consumes a significant portion of the total computational resources, which accounts for approximately 70\% of the total FEs.
Specifically, each dimension variable requires approximately 40 FEs to search for the minimum point. 


Overall, the minimum points search stage consumes a significant portion of the total computational resources.
The computational resources consumed in this phase are influenced by the problem dimension but remain unaffected by the separability type.
However, due to the composite nature of CSG, the variables identified in the previous stage can help reduce the dimension of problem decomposition in the subsequent stage.  
As a result, the computational resources consumption of the entire algorithm can be influenced by the variables identified in each stage.

In the first stage of CSG, it costs 2$n$+2 FEs to find all $p$ additively separable variableses, so there are $n \-- p$ variables need to be detected in the subsequent stages.
Each dimensional variable requires eight FEs for a multiplicative separability test, which imposes a relatively large computational burden.
However, the historical information from the additively separable detection stage can be reused, reducing the number of FEs in the second stage by half. 
Therefore, the FEs in the second stage to examine whether the remaining $n \-- p$ variables are multiplicatively separable is $4 \times (n \-- p)$.
Then the remaining $n - p - q$ variables need to be searched for their minimum points.

There are two termination conditions for the minimum points search using the golden section method. 
The first condition is when the function value of the two sampling points is equal. 
The second condition is when the difference between the distance of the two sampling points is less than the specified search accuracy $\epsilon$.

The second condition is influenced by the range of variable values and the precision of the GSS method.
The larger the range of variable values and the higher the search precision, the more FEs are required to satisfy the termination condition.
In the GSS method, the search interval is reduced to ${(\sqrt{5}-1)}/{2}$ of its original value after each iteration.
If the upper bound of a variable is $ub$, the lower bound is $lb$, and the precision of the golden section search is $\epsilon$, then the maximum number of FEs required to search for the minimum point of the variable is given by the formula:
\begin{eqnarray}
\begin{split}
\begin{aligned}
k = \lceil \log_{a} {\frac{\epsilon}{(ub-lb)}} x \rceil
\end{aligned}
\end{split}
\end{eqnarray}
where $a={(\sqrt{5}-1)}/{2}$.
Therefore, the search for the minimum points of the remaining $ n \-- p \--q $ variables requires $k\times (n \-- p \--q)$ FEs.

After the GSS stage, each variable requires three FEs to determine if it is a generally separable variable. 
Therefore, all variables require $3 \times (n - p - q)$ FEs.
The type of non-separable variables in the problem, which includes the presence of indirect interactions and the number of non-separable variable groups, influences the computational resource consumption of grouping them in CSG.

After analyzing the computational resources required for each stage, we consider the problems without separable variables to derive the computational complexity.
The computational complexity of each component of CSG is listed in Table \ref{com}.
Overall, the computational complexity of CSG when utilized to decompose an $n$-dimensional problem is $\mathcal{O}(nlogn)$.

\begin{table}[htbp]
\captionsetup{labelfont={color=black}} 
  \centering
  \caption{{Computational complexity of each component of CSG}}
       \renewcommand{\arraystretch}{1.2}
    \begin{tabular}{cc}
    \toprule
    Component     &  Computational Complexity \\
    \midrule
    Separable Sariables Detection      &  $\mathcal{O}(9n)$\\
    Golden Section Search     &  $\mathcal{O}(kn)$\\
    Non-separable Variable Grouping     & $\mathcal{O}(nlogn)$ \\
    \bottomrule
    \end{tabular}%
  \label{com}%
  \vspace{-0cm}
\end{table}%
\color{black}

{In summary, GSG utilizes the same approach to address different types of variables, including searching for the minimum point of each dimensional variable and conducting interaction detection.  
In contrast, CSG leverages the potential of the DG method to efficiently decompose additively separable variables.   
Furthermore, we introduce two novel methods to effectively handle multiplicative separable variables and groups of non-separable variables.
Additionally, CSG further establishes a step-by-step decomposition framework, which greatly improves the efficiency of problem decomposition.
The experimental study in the following section of the article also includes a comparison of the computational resource consumption between GSG and CSG on several benchmarks.}

\section{Experimental Study}\label{section Experimental}
In this section, we first design 15 benchmark functions to thoroughly evaluate the performance of CSG, along with some state-of-the-art decomposition methods. 
We then examine the grouping accuracy and efficiency of CSG and compare it with other methods. 
Finally, we integrate the grouping results into the CC framework to test the optimization performance.

For comparison, we select GSG \cite{GSG}, RDG2 \cite{RDG2}, and MDG \cite{MDG} as the comparison methods for CSG. 
GSG serves as the main optimization target for CSG. 
RDG2 utilizes a recursive decomposition grouping method with an adaptive threshold, making it efficient and accurate on most benchmark functions. 
However, the RDG2 method may mistakenly assign some separable variables into the same group if it fails to recognize their separability. 
{The reason for choosing RDG2 instead of RDG3 as the comparative algorithm is that RDG3 is specifically designed for addressing overlapping problems.}
On the other hand, MDG merges non-separable variables into appropriate groups using a binary-tree-based iterative method, which may result in more non-separable variable groups when the grouping accuracy is not high.

CSG adopts the same method as RDG2 to set the threshold $\epsilon_{1}$ for detecting additively separable variables and designs threshold $\epsilon_{2}$ according to MDG for detecting multiplicatively separable variables. 
The parameters used in the comparison algorithms are kept the same as in their original papers to ensure fairness in the evaluation process.

\begin{table*}[!t]
  \caption{Test functions in the proposed BMS}
    \resizebox{\linewidth}{!}{
     \renewcommand{\arraystretch}{0.9}
 \renewcommand{\tabcolsep}{13pt}
    \begin{tabular}{lllllllllll}
    \toprule
    No. & \multicolumn{2}{p{8.38em}}{Category} & \multicolumn{8}{l}{Test function} \\
    \midrule
    $f_{1}$ & \multicolumn{2}{l}{\multirow{6}[2]{*}{\shortstack{Fully separable functions with \\ two types of separable variables}}} & \multicolumn{8}{l}{$F_{1}(\emph{\textbf{x}}) = F_{\mathrm{rast}}(\textbf{\emph{z}}(P_1:P_{D/2})) + F_{\mathrm{prodras}}(\textbf{\emph{z}}(P_{{D/2}+1}:P_D))$     } \\
    $f_{2}$ & \multicolumn{2}{l}{} & \multicolumn{8}{l}{$F_{2}(\emph{\textbf{x}}) = F_{\mathrm{sphe}}(\textbf{\emph{z}}(P_1:P_{D/2})) + F_{\mathrm{prodsqu}}(\textbf{\emph{z}}(P_{{D/2}+1}:P_D))$} \\
    $f_{3}$ & \multicolumn{2}{l}{} & \multicolumn{8}{l}{$F_{3}(\emph{\textbf{x}}) = F_{\mathrm{rast}}(\textbf{\emph{z}}(P_1:P_{D/2})) + F_{\mathrm{logabs}}(\textbf{\emph{z}}(P_{{D/2}+1}:P_D))$} \\
    $f_{4}$ & \multicolumn{2}{l}{} & \multicolumn{8}{l}{$F_{4}(\emph{\textbf{x}}) = F_{\mathrm{sphe}}(\textbf{\emph{z}}(P_1:P_{D/2})) + F_{\mathrm{cone}}(\textbf{\emph{z}}(P_{{D/2}+1}:P_D))$} \\
    $f_{5}$ & \multicolumn{2}{l}{} & \multicolumn{8}{l}{$F_{5}(\emph{\textbf{x}}) = F_{\mathrm{prodsqu}}(\textbf{\emph{z}}(P_1:P_{D/2})) + F_{\mathrm{logabs}}(\textbf{\emph{z}}(P_{{D/2}+1}:P_D))$} \\
    $f_{6}$ & \multicolumn{2}{l}{} & \multicolumn{8}{l}{$F_{6}(\emph{\textbf{x}}) = F_{\mathrm{prodras}}(\textbf{\emph{z}}(P_1:P_{D/2})) + F_{\mathrm{cone}}(\textbf{\emph{z}}(P_{{D/2}+1}:P_D))$} \\
    \midrule
    $f_{7}$ & \multicolumn{2}{l}{\multirow{3}[2]{*}{\shortstack{Fully separable functions with \\ three types of separable variables}}} & \multicolumn{8}{l}{$F_{7}(\emph{\textbf{x}}) = F_{\mathrm{rast}}(\textbf{\emph{z}}(P_1:P_{D/5*2})) + F_{\mathrm{prodsqu}}(\textbf{\emph{z}}(P_{D/5*2+1}:P_{D/10*7})) + F_{\mathrm{logabs}}(\textbf{\emph{z}}(P_{D/10*7+1}:P_{D}))$}   \\
    $f_{8}$ & \multicolumn{2}{l}{} & \multicolumn{8}{l}{$F_{8}(\emph{\textbf{x}}) = F_{\mathrm{elli}}(\textbf{\emph{z}}(P_1:P_{D/10*3})) + F_{\mathrm{prodras}}(\textbf{\emph{z}}(P_{D/10*3+1}:P_{D/10*7})) + F_{\mathrm{cone}}(\textbf{\emph{z}}(P_{D/10*7+1}:P_{D}))$}   \\
    $f_{9}$ & \multicolumn{2}{l}{} & \multicolumn{8}{l}{$F_{9}(\emph{\textbf{x}}) = F_{\mathrm{sphe}}(\textbf{\emph{z}}(P_1:P_{D/10*3})) + F_{\mathrm{prodras}}(\textbf{\emph{z}}(P_{D/10*3+1}:P_{D/5*3})) + F_{\mathrm{logabs}}(\textbf{\emph{z}}(P_{D/5*3+1}:P_{D}))$}   \\
    \midrule
    $f_{10}$ & \multicolumn{2}{l}{\multirow{7}[2]{*}{\shortstack{Partially separable functions with \\ three types of separable variables}}} & \multicolumn{8}{l}{$F_{10}(\emph{\textbf{x}}) = F_{\mathrm{sphe}}(\textbf{\emph{z}}) + F_{\mathrm{prodras}}(\textbf{\emph{z}}) + F_{\mathrm{cone}}(\textbf{\emph{z}}) + F_{\mathrm{rosen}}(\textbf{\emph{z}}(P_{1}:P_{D/4}))$ }   \\
    $f_{11}$ & \multicolumn{2}{l}{} & \multicolumn{8}{l}{$F_{11}(\emph{\textbf{x}}) = F_{\mathrm{rast}}(\textbf{\emph{z}}) + F_{\mathrm{prodsqu}}(\textbf{\emph{z}}) + F_{\mathrm{logabs}}(\textbf{\emph{z}}) + F_{\mathrm{schw}}(\textbf{\emph{z}}(P_{1}:P_{D/4})) $}   \\
    $f_{12}$ & \multicolumn{2}{l}{} & \multicolumn{8}{l}{$F_{12}(\emph{\textbf{x}}) = F_{\mathrm{rast}}(\textbf{\emph{z}}) + F_{\mathrm{prodsqu}}(\textbf{\emph{z}}) + F_{\mathrm{logabs}}(\textbf{\emph{z}}) + \sum_{k=1}^{D/4/m}F_{\rm{rot\_}\rm{rast}}(\textbf{\emph{z}}(P_{(k-1)*m+1}:P_{k*m})) $}   \\
    $f_{13}$ & \multicolumn{2}{l}{} & \multicolumn{8}{l}{$F_{13}(\emph{\textbf{x}}) = F_{\mathrm{rast}}(\textbf{\emph{z}}) + F_{\mathrm{prodsqu}}(\textbf{\emph{z}}) + F_{\mathrm{logabs}}(\textbf{\emph{z}}) + \sum_{k=1}^{D/4/m}F_{\rm{schw}}(\textbf{\emph{z}}(P_{(k-1)*m+1}:P_{k*m})) ) $}   \\
    $f_{14}$ & \multicolumn{2}{l}{} & \multicolumn{8}{l}{$F_{14}(\emph{\textbf{x}}) = F_{\mathrm{sphe}}(\textbf{\emph{z}}) + F_{\mathrm{prodras}}(\textbf{\emph{z}}) + F_{\mathrm{cone}}(\textbf{\emph{z}}) + \sqrt{ \sum_{k=1}^{D/4/m}F_{\rm{rot\_}\rm{rast}}[\textbf{\emph{z}}(P_{(k-1)*m+1}:P_{k*m})]}$}   \\
    $f_{15}$ & \multicolumn{2}{l}{} & \multicolumn{8}{l}{$F_{15}(\emph{\textbf{x}}) = F_{\mathrm{rast}}(\textbf{\emph{z}}) + F_{\mathrm{prodsqu}}(\textbf{\emph{z}}) + F_{\mathrm{logabs}}(\textbf{\emph{z}}) + \log \{ \sum_{k=1}^{D/4/m}F_{schw}(\textbf{\emph{z}}(P_{(k-1)*m+1}:P_{k*m}))\} $}   \\
    \bottomrule
    \end{tabular}%
  \label{tab BMS}%
}
\vspace{-0.5cm}
\end{table*}%

\subsection{Benchmark Functions}
Based on six basic functions, CEC benchmark sets \cite{CEC, CEC2010, CEC2013} generate a range of separable functions, partially-separable functions, and fully-non-separable functions.
However, most separable functions in the CEC test functions are additively separable functions, with Ackley's function being the only non-additively separable function among the basic functions. 
This limitation makes it difficult to evaluate the performance of decomposition methods that focus on identifying various types of separable variables.

To address these limitations, the BNS \cite{GSG} introduces four non-additively separable basis functions, including two multiplicatively separable basis functions and two composite separable basis functions. 
Based on these basis functions, BNS designs 12 test problems with varying degrees of separability.

However, a drawback of BNS is that each function only contains one type of separable variable, while real-world problems often involve multiple types of separable variables. 
{In the context of NNs, the pre-activation function encompasses both multiplicative separable variable weights and additive separable variable biases \cite{DDG}.
Moreover, activation functions in neural networks generally exhibit the monotonicity property. 
Therefore, the separable variables within its inner function satisfy composite separability \cite{GSG}.
In the context of system reliability analysis problems, the reliability of a series system is the product of the reliability of multiple subsystems \cite{real1}.     
The human health model \cite{real2} and business optimal control model \cite{real3} involve both additive separability and multiplicative separability between multiple submodules.}
To overcome this limitation, we propose a novel benchmark called the benchmark based on multiple separability (BMS), which incorporates diverse types of separable variables within each benchmark function. 
The functions in BMS are generated by combining basis functions, resulting in a highly scalable approach.
The BMS benchmark is presented in Table \ref{tab BMS}.

The basis functions utilized in BMS are derived from CEC \cite{CEC2013} and BNS \cite{GSG}. 
The benchmark functions employed in this article have dimensions of 1000, 2000, 3000, 4000, and 5000, with the problem dimension denoted as $D$.
In the formulation, $P$ represents a sequence of random perturbations ranging from 1 to $D$, which signifies the permutation of the $D$-dimensional variables in the problem. 
Additionally, $\textbf{\emph{z}}= \textbf{\emph{x}} - \textbf{\emph{o}}$ denotes shifting the global optimum of variables from 0 to $\textbf{\emph{o}}$. 
The subscript ``$rot$'' indicates the coordinate rotation operator, which utilizes an orthogonal matrix to transform the corresponding variables and make the problems non-separable.

$f_1$-$f_6$ are fully separable functions that consist of two types of separable variables, while $f_7$-$f_9$ are fully separable functions with three types of separable variables.
$f_{10}$-$f_{15}$ are partially separable functions with three types of separable variables and non-separable variables.
The number of separable variables for the three types of separability and non-separable variables in $f_{10}$-$f_{15}$ is $D$/4, which is not shown in Table \ref{tab BMS} to avoid excessively long formulas.
$f_{10}$-$f_{15}$ incorporate non-separable variables that exhibit various forms of dependence.
$f_{10}$-$f_{11}$ contains only one non-separable variable group, while the subcomponents in $f_{12}$-$f_{15}$ have a size of 50.


In summary, the functions within BMS encompass a wide range of separable variables and groups of non-separable variables, introducing difficulties in problem decomposition and optimization.

\subsection{Comparison on Decomposition}
In this section, we will compare the performance of CSG with GSG, RDG2, and MDG in decomposing benchmark functions within BMS. 

\subsubsection{Accuracy Metric}
The evaluation metrics will include the accuracy of the decomposition and the computational resource consumption.

For the original problem, the set of truly separable variables is denoted as $\textbf{\emph{sep}}^*$, and the $m$ groups of true non-separable variables are denoted as $\textbf{\emph{nonsep}}^*=\{\textbf{\emph{g}}_i^*,...,\textbf{\emph{g}}_m^*\}$.
Regarding the grouping results, the formed set of separable variables is denoted as $\textbf{\emph{sep}}$ and the formed $k$ groups of non-separable variables are denoted as $\textbf{\emph{nonsep}}=\{\textbf{\emph{g}}_i,...,\textbf{\emph{g}}_k\}$.

The grouping accuracy metric for the separable variables (SA) in the problem is defined as:
\begin{eqnarray}\label{equ SA}
\begin{split}
SA=\frac{|\textbf{\emph{sep}}^*\cap \textbf{\emph{sep}}|}{|\textbf{\emph{sep}}^*|}
\end{split}
\end{eqnarray}
This metric represents the proportion of the obtained separable variables in the true separable variables set.

The grouping accuracy metric for the non-separable variables (NA) in the problem is defined as:
\begin{eqnarray}\label{equ NA}
\begin{split}
NA=
\frac{\sum_i^m \max\{|\textbf{\emph{g}}_i^*\cap \textbf{\emph{g}}_1|,|\textbf{\emph{g}}_i^*\cap \textbf{\emph{g}}_2|,..,|\textbf{\emph{g}}_i^*\cap \textbf{\emph{g}}_k|\}}{\sum_i^m |\textbf{\emph{g}}_i^*|}
\end{split}
\end{eqnarray}
For each true non-separable variable group $\textbf{\emph{g}}_i^*$, we find the group $\textbf{\emph{g}}_i$ that contains the highest number of common variables among all the formed non-separable variable groups.
The group $\textbf{\emph{g}}_i$ is only checked once.
We calculate the sum of common variables for each true non-separable variable group and its most similar formed group. 
NA represents the proportion of the sum of common variables in the true non-separable variables.

\begin{table*}[h]
  \centering
  \caption{The decomposition results of each algorithm on 1000-D BMS. SA represents the decomposition accuracy for separable variables, while NA represents the decomposition accuracy for non-separable variables.}
  \resizebox{\linewidth}{!}{
     \renewcommand{\arraystretch}{1}
 \renewcommand{\tabcolsep}{10pt}
\begin{tabular}{ccc|ccc|ccc|ccc|ccc}
\hline
\multicolumn{2}{c}{\multirow{2}[4]{*}{Benchmark}} & \multirow{2}[4]{*}{Problem} & \multicolumn{3}{c|}{CSG} & \multicolumn{3}{c|}{GSG} & \multicolumn{3}{c|}{RDG2} & \multicolumn{3}{c}{MDG} \bigstrut\\
\cline{4-15}\multicolumn{2}{c}{} &       & SA    & NA    & FEs   & SA    & NA    & FEs   & SA    & NA    & FEs   & SA    & NA    & FEs \bigstrut\\
\hline
\multicolumn{2}{c}{\multirow{13}[6]{*}{BMS}} & $f_{1}$ & \textbf{100.0\%} & -    & 4002 & \textbf{100.0\%} & -     & 38746 & 50.0\% & -     & 6610  & 50.0\% & -     & 2501 \bigstrut[t]\\
\multicolumn{2}{c}{} & $f_{2}$ & \textbf{100.0\%} & -     & 4002 & \textbf{100.0\%} & -     & 37238 & 50.0\% & -     & 6586  & 50.0\% & -     & 2501 \\
\multicolumn{2}{c}{} & $f_{3}$ & \textbf{100.0\%} & -     & 25993 & \textbf{100.0\%} & -     & 38989 & 50.0\% & -     & 6667  & 50.0\% & -     & 2981 \\
\multicolumn{2}{c}{} & $f_{4}$ & \textbf{100.0\%} & -     & 20271 & \textbf{100.0\%} & -     & 32970 & 50.0\% & -     & 7063  & 50.0\% & -     & 2505 \\
\multicolumn{2}{c}{} & $f_{5}$ & \textbf{100.0\%} & -     & 28002 & \textbf{100.0\%} & -     & 37353 & 0.0\%  & -     & 8170   & 0.0\% & -     & 6178 \\
\multicolumn{2}{c}{} & $f_{6}$ & \textbf{100.0\%} & -     & 22088 & \textbf{100.0\%} & -     & 34997 & 0.1\%  & -     & 10030  & 0  \% & -     & 5876 \\
\multicolumn{2}{c}{} & $f_{7}$ & \textbf{100.0\%} & -     & 17602 & \textbf{100.0\%} & -     & 37661 & 40.0\% & -     & 8680  & 40.0\% & -     & 4517 \\
\multicolumn{2}{c}{} & $f_{8}$ & \textbf{100.0\%} & -     & 13720 & \textbf{100.0\%} & -     & 32663 & 30.0\% & -     & 9805  & 30.0\% & -     & 4688 \\
\multicolumn{2}{c}{} & $f_{9}$ & \textbf{100.0\%} & -     & 22402 & \textbf{100.0\%} & -     & 38019 & 30.0\% & -     & 8734  & 30.0\% & -     & 5034 \\
\multicolumn{2}{c}{} & $f_{10}$ &\textbf{100.0\%}   & \textbf{100.0\%}  & 31269 &\textbf{100.0\%}  & \textbf{100.0\%} & 49252  & 33.3\%  & \textbf{100.0\%}  & 20590    & 33.3\%  & \textbf{100.0\%} & 8194 \\
\multicolumn{2}{c}{} & $f_{11}$ &\textbf{100.0\%}   & \textbf{100.0\%}  & 25716 &\textbf{100.0\%}  & \textbf{100.0\%} & 40889  & 33.3\%  & \textbf{100.0\%}  & 10492    & 33.3\%  & \textbf{100.0\%} & 5904 \\
\multicolumn{2}{c}{} & $f_{12}$ &\textbf{100.0\%}   & \textbf{100.0\%}  & 30367 &\textbf{100.0\%}  & \textbf{100.0\%} & 44458  & 33.3\%  & \textbf{100.0\%}  & 12730    & 33.3\%  & \textbf{100.0\%} & 7086 \\
\multicolumn{2}{c}{} & $f_{13}$ &\textbf{100.0\%}   & \textbf{100.0\%}  & 30963 &\textbf{100.0\%}  & \textbf{100.0\%} & 45138  & 33.3\%  & \textbf{100.0\%}  & 12766    & 33.3\%  & \textbf{100.0\%} & 7091 \\
\multicolumn{2}{c}{} & $f_{14}$ &\textbf{100.0\%}   & \textbf{100.0\%}  & 27506 &\textbf{100.0\%}  & \textbf{100.0\%} & 42589  & 33.3\%  & 20.0\%  & 11137    & 33.3\%  & 20.0\% & 5769 \\
\multicolumn{2}{c}{} & $f_{15}$ &\textbf{100.0\%}   & \textbf{100.0\%}  & 29462 &\textbf{100.0\%}  & \textbf{100.0\%} & 44105  & 33.3\%  & 20.0\%  & 10306    & 33.3\%  & 20.0\% & 5892 \\
\hline
\end{tabular}%
}
  \label{tab GroupResults}%
\end{table*}%

\begin{table*}[!h]
  \centering
  \caption{The decomposition results of each algorithm on higher-dimensional problems in BMS. The results represent the average decomposition accuracy and average computational resource consumption for all the problems in the benchmarks.}
  \resizebox{\linewidth}{!}{
     \renewcommand{\arraystretch}{1.1}
 \renewcommand{\tabcolsep}{6pt}
\begin{tabular}{cc|ccc|ccc|ccc|ccc}
\hline
\multirow{2}[4]{*}{Benchmark} & \multirow{2}[4]{*}{Dimensionality} & \multicolumn{3}{c|}{CSG} & \multicolumn{3}{c|}{GSG} & \multicolumn{3}{c|}{RDG2} & \multicolumn{3}{c}{MDG} \bigstrut\\
\cline{3-14}      &       & SA-Aver & NA-Aver & FEs-Aver & SA-Aver & NA-Aver & FEs-Aver & SA-Aver & NA-Aver & FEs-Aver & SA-Aver & NA-Aver & FEs-Aver \bigstrut\\
\hline
\multirow{5}[2]{*}{BMS} & 1000  & \textbf{100.0\%} & \textbf{100.0\%} & 22224 & \textbf{100.0\%} & \textbf{100.0\%}  & 39671  & 33.3\% & 73.3\% & 10024  & 33.3\% & 73.3\% & 5114  \bigstrut[t]\\
      & 2000  & \textbf{100.0\%} & \textbf{100.0\%} & 45389 & \textbf{100.0\%} & \textbf{100.0\%}  & 80836  & 33.4\% & 70.0\% & 20709  & 33.3\% & 70.0\% & 10446  \\
      & 3000  & \textbf{100.0\%} & \textbf{100.0\%} & 67492 & \textbf{100.0\%} & \textbf{100.0\%}  & 121048  & 33.5\% & 68.9\% & 32038  & 33.3\% & 68.9\% & 15842    \\
      & 4000  & \textbf{100.0\%} & \textbf{100.0\%} & 92735 & \textbf{100.0\%} & 90.5\%  & 160865  & 33.3\% & 68.3\% & 42817  & 33.3\% & 68.3\% & 21262     \\
      & 5000  &\textbf{100.0\%} & \textbf{100.0\%} & 115758 & \textbf{100.0\%} & 93.8\%  & 201110  & 33.3\% & 68\% & 54424  & 33.3\% & 68\% & 26693 \bigstrut[b]\\
\hline
\end{tabular}%

    }
  \label{tab GroupResults1}%
  \vspace{-0.5cm}
\end{table*}%

\subsubsection{Decomposition on 1000-dimensional BMS Problems}
The decomposition experiment results on 1000-dimensional problems are presented in Table \ref{tab GroupResults}.

Both CSG and GSG effectively group all variables of the 15 benchmark functions within BMS.
However, there is a difference in CSG and GSG. 
CSG categorizes separable variables into additively separable, multiplicatively separable, and generally separable variables, while GSG groups all separable variables into a single group.

For functions $f_1$-$f_9$, which are fully separable, RDG2 and MDG can accurately identify all additively separable variables, but incorrectly classify multiplicatively and composite separable variables as non-separable variables. 
The SA of RDG2 and MDG for $f_1$-$f_{15}$ represents the proportion of additively separable variables to the total separable variables.

When dealing with separable variables, binary-based RDG2 methods typically group multiple non-additively separable variable groups as a single group, whereas MDG may generate multiple groups of variables.
This difference may be attributed to merge-based grouping methods tending to merge multiple separable variables into a group when they are unable to handle certain types of separable variables.
When the decomposing-based grouping method cannot identify some separable variables, the size of the variable group remains unchanged.

Functions $f_{10}$-$f_{13}$ contain multiple independent partially additively separable subcomponents, except for separable variables, and RDG2 and MDG can accurately group them with 100\% grouping accuracy.
However, RDG2 and MDG fail to identify groups of partially non-additively separable variables in functions $f_{14}$-$f_{15}$.
When dealing with non-separable variables in $f_{14}$-$f_{15}$, RDG2 and MDG tend to identify multiple partially non-additively separable variable groups as a single non-separable variable group.

\subsubsection{Decomposition on Higher-dimensional BMS Problems}
To test the scalability of CSG, we compare the decomposition accuracy and computational resource consumption of CSG, GSG, RDG2, and MDG on the BMS benchmark sets with dimensions ranging from 1000 to 5000. 
The decomposition results for higher dimensions are shown in Table \ref{tab GroupResults1}.

The results show that CSG can accurately decompose problems with dimensions ranging from 1000 to 5000.
CSG employs a computationally efficient approach to identify separable variables, resulting in an average computational resource consumption that is about half of that of GSG on BMS functions with dimensions ranging from 1000 to 5000.

When dealing with non-separable variables in high-dimensional problems, GSG does not accurately decompose all problems.
This may be because high-dimensional problems exhibit more intricate properties and possess a more complex fitness landscape.
Furthermore, when GSG is used to handle high-dimensional problems, the parameter settings may need further consideration. 
However, CSG effectively reduces the complexity of the problem by separating the separable variables beforehand, resulting in a lower-dimensional problem consisting only of non-separable variables.
CSG achieves higher accuracy than GSG when grouping non-separable variables in high-dimensional problems.

RDG2 and MDG also exhibit lower accuracy than CSG when dealing with high-dimensional problems. 
This is because RDG2 and MDG can only decompose additively separable problems.

\subsubsection{Decomposition on Classical Problems}
In addition to the BMS benchmark, we have conducted comparative experiments on the CEC and BNS benchmarks, both having a dimensionality of 1000.
Due to the page limit, detailed decomposition results for the CEC and BNS benchmarks are presented in Section S.II of the Supplementary Material.

Overall, CSG consumes fewer computational resources than those by GSG on 32 out of a total of 47 benchmark problems.
Since CSG utilizes a fixed amount of computational resources to detect separable variables, it consumes more computational resources than GSG does on the problems without separable variables.
MDG and RDG2 can efficiently decompose additively separable problems, but both fail to effectively address non-additively separable problems in BNS.

\subsubsection{Computational Resource Consumption}
We use the computational resource consumption of BMS functions with 1000 dimensions as an example.
The cutting-edge DG-series algorithms, which are based on the finite difference principle, have significantly reduced their computational costs. 
RDG2 requires approximately $10^3$ FEs for fully separable problems such as $f_1$-$f_9$ and over ten thousand FEs for partially separable problems like $f_{10}$-$f_{15}$.
However, MDG requires around $10^3$ FEs for all problems in BMS.

For 1000-dimensional problems in BMS, GSG requires an average of about $4 \times 10^4$ FEs, with three-quarters of the FEs used in the minimum points search phase.
In contrast, CSG consumes approximately $2.4 \times 10^4$ FEs on average for variable grouping, which is roughly half of the computational resource consumption of GSG.
Therefore, the computational complexity of CSG is significantly lower than that of GSG does.

\begin{table} [htbp]
  \centering
  \caption{The computational resource consumption for grouping non-separable variables on $f_{10}$-$f_{15}$ in BMS}
  \tabcolsep=0.5cm
  \renewcommand{\arraystretch}{1.0} 
    \begin{tabular}{cccc}
    \hline
    Benchmark & Problems & CSG   & CSG-GSG \bigstrut\\
    \hline
    \multirow{6}[2]{*}{BMS} & $f_{10}$   & 10443 & 13428 \bigstrut[t]\\
          &  $f_{11}$     & 2247  & 5508 \\
          &  $f_{12}$     & 5769  & 6900 \\
          &  $f_{13}$     & 5763  & 8472 \\
          &  $f_{14}$     & 5751  & 7038 \\
          &  $f_{15}$     & 5625  & 6660 \bigstrut[b]\\
    \hline
    \end{tabular}%
  \label{tab FEsForNSV}%
\vspace{-0.5cm}
\end{table}%

To evaluate the efficiency of CSG and GSG in grouping non-separable variables, a comparative experiment is designed using benchmark functions $f_{10}$-$f_{15}$, which contain non-separable variables. 
This experiment aims to assess the performance of NVG. 
After performing the minimum points search and identifying generally separable variables using CSG, GSG is employed to group the remaining non-separable variables.
This combined approach is denoted as CSG-GSG.
The number of FEs required by CSG and CSG-GSG for grouping non-separable variables following the minimum points search stage is specifically measured. 
The computational resources needed by CSG and CSG-GSG for grouping non-separable variables are statistically obtained, and the results are presented in Table \ref{tab FEsForNSV}. 
The number of FEs required by CSG to decompose $f_{10}$-$f_{15}$ consistently outperforms CSG-GSG. 
These results demonstrate that CSG has lower computational complexity than CSG-GSG in grouping non-separable variables, indicating that NVG is an efficient approach for this task.


\begin{table*}[t]
  \centering
  \caption{The optimization results of various grouping methods on 1000-D BMS problems over 25 independent runs. Three subtables record the optimization results for fes of $1.2 \times 10^5$, $6 \times 10^5$ and $3 \times 10^6$, respectively.}
  \resizebox{\linewidth}{!}{
       \renewcommand{\arraystretch}{01}
 \renewcommand{\tabcolsep}{10pt}
\begin{tabular}{cc|ccc|ccc|ccc|ccc}
\hline
\multicolumn{14}{c}{FEs = $1.2 \times 10^5$}  \bigstrut\\
\hline
\multirow{2}[4]{*}{Benchmark} & \multirow{2}[4]{*}{Problem} & \multicolumn{3}{c|}{CSG} & \multicolumn{3}{c|}{GSG} & \multicolumn{3}{c|}{RDG2} & \multicolumn{3}{c}{MDG} \bigstrut\\
\cline{3-14}      &       & Mean  & Median & Std   & Mean  & Median & Std   & Mean  & Median & Std   & Mean  & Median & Std \bigstrut\\
\hline
\multirow{14}[2]{*}{BMS} & $f_{1}$ & \textbf{3.47E+03} & \textbf{3.47E+03} & \textbf{5.40E+01} & 4.09E+03 & 4.09E+03 & 6.53E+01 & 4.35E+03 & 4.32E+03 & 1.32E+02 & 4.22E+03 & 4.24E+03 & 1.44E+02  \bigstrut[t]\\
      & $f_{2}$ & \textbf{2.23E+02} & \textbf{2.24E+02} & \textbf{4.98E+00} & 4.26E+02 & 4.25E+02 & 9.92E+00 & 6.17E+02 & 6.12E+02 & 4.96E+01 & 5.75E+02 & 5.74E+02 & 5.21E+01  \\
      & $f_{3}$ & \textbf{2.96E+03} & \textbf{2.96E+03} & \textbf{5.23E+01} & 3.16E+03 & 3.16E+03 & 6.00E+01 & 3.36E+03 & 3.35E+03 & 6.28E+01 & 6.08E+03 & 6.10E+03 & 1.03E+02 \\
      & $f_{4}$    & \textbf{5.95E+03} & \textbf{5.96E+03} & \textbf{1.39E+02} & 7.94E+03 & 7.97E+03 & 1.69E+02 & 4.66E+04 & 4.66E+04 & 1.80E+03 & 4.67E+04 & 4.69E+04 & 2.23E+03 \\
      & $f_{5}$    & 5.57E+01 & 5.57E+01 & 2.82E-01 & 5.80E+01 & 5.79E+01 & 2.07E-01 & \textbf{5.17E+01} & \textbf{5.17E+01} & \textbf{6.31E-01} & 1.14E+02 & 1.15E+02 & 5.40E+00 \\
      & $f_{6}$    & \textbf{5.40E+03} & \textbf{5.40E+03} & \textbf{1.68E+02} & 7.28E+03 & 7.27E+03 & 1.94E+02 & 7.40E+03 & 7.30E+03 & 6.67E+02 & 6.82E+03 & 6.74E+03 & 5.22E+02 \\
      & $f_{7}$    & 8.72E+03 & 8.72E+03 & 6.47E+01 & 9.02E+03 & 9.00E+03 & 6.55E+01 & \textbf{8.44E+03} & \textbf{8.44E+03} & \textbf{7.02E+01} & 1.10E+04 & 1.10E+04 & 9.50E+01 \\
      & $f_{8}$    & \textbf{5.35E+05} & \textbf{5.36E+05} & \textbf{1.60E+04} & 8.94E+05 & 8.95E+05 & 3.31E+04 & 1.69E+06 & 1.67E+06 & 7.42E+04 & 1.61E+06 & 1.60E+06 & 8.15E+04 \\
      & $f_{9}$    & \textbf{7.39E+03} & \textbf{7.40E+03} & \textbf{4.33E+01} & 7.68E+03 & 7.69E+03 & 4.82E+01 & 9.43E+03 & 9.39E+03 & 2.49E+02 & 2.46E+04 & 2.43E+04 & 2.03E+03 \\
      & $f_{10}$   & \textbf{4.71E+03} & \textbf{4.70E+03} & \textbf{1.77E+02} & 6.71E+03 & 6.71E+03 & 1.78E+02 & 1.55E+04 & 1.54E+04 & 1.09E+03 & 1.32E+04 & 1.31E+04 & 8.51E+02 \\
      & $f_{11}$   & 7.18E+03 & 7.19E+03 & 5.37E+01 & 7.42E+03 & 7.43E+03 & 6.49E+01 & \textbf{6.81E+03} & \textbf{6.81E+03} & \textbf{4.87E+01} & 8.64E+03 & 8.65E+03 & 7.16E+01 \\
      & $f_{12}$   & 1.01E+04 & 1.01E+04 & 7.38E+01 & 1.04E+04 & 1.04E+04 & 7.79E+01 & \textbf{9.83E+03} & \textbf{9.81E+03} & \textbf{7.55E+01} & 1.12E+04 & 1.13E+04 & 8.53E+01 \\
      & $f_{13}$   & 9.17E+03 & 9.17E+03 & 1.09E+02 & 9.59E+03 & 9.61E+03 & 1.52E+02 & \textbf{8.57E+03} & \textbf{8.55E+03} & \textbf{1.03E+02} & 1.05E+04 & 1.05E+04 & 1.54E+02 \\
      & $f_{14}$   & \textbf{3.39E+03} & \textbf{3.39E+03} & \textbf{8.66E+01} & 4.55E+03 & 4.54E+03 & 1.50E+02 & 1.35E+04 & 1.36E+04 & 9.86E+02 & 1.23E+04 & 1.21E+04 & 8.40E+02 \\
      & $f_{15}$   & 7.52E+03 & 7.50E+03 & 6.37E+01 & 7.75E+03 & 7.75E+03 & 5.26E+01 & \textbf{6.80E+03} & \textbf{6.79E+03} & \textbf{6.52E+01} & 8.48E+03 & 8.48E+03 & 6.94E+01 \\
\hline
\multicolumn{2}{c|}{Number of firsts} & \multicolumn{3}{c|}{9} & \multicolumn{3}{c|}{0} & \multicolumn{3}{c|}{6} & \multicolumn{3}{c}{0} \bigstrut\\
\hline
\hline
\multicolumn{14}{c}{FEs = $6 \times 10^5$}  \bigstrut\\
\hline
\multirow{2}[4]{*}{Benchmark} & \multirow{2}[4]{*}{Problem} & \multicolumn{3}{c|}{CSG} & \multicolumn{3}{c|}{GSG} & \multicolumn{3}{c|}{RDG2} & \multicolumn{3}{c}{MDG} \bigstrut\\
\cline{3-14}      &       & Mean  & Median & Std   & Mean  & Median & Std   & Mean  & Median & Std   & Mean  & Median & Std \bigstrut\\
\hline
\multirow{14}[2]{*}{BMS} & $f_{1}$ & \textbf{1.12E+03} & \textbf{1.12E+03} & \textbf{2.33E+01} & 1.19E+03 & 1.19E+03 & 2.91E+01 & 1.37E+03 & 1.38E+03 & 4.08E+01 & 1.25E+03 & 1.25E+03 & 3.87E+01 \bigstrut[t]\\
    & $f_{2}$    & \textbf{1.00E+02} & \textbf{1.00E+02} & \textbf{5.57E-05} & 1.00E+02 & 1.00E+02 & 1.26E-04 & 1.24E+02 & 1.23E+02 & 2.58E+00 & 1.25E+02 & 1.25E+02 & 3.75E+00 \\
    & $f_{3}$    & \textbf{9.11E+02} & \textbf{9.13E+02} & \textbf{2.29E+01} & 9.31E+02 & 9.33E+02 & 3.18E+01 & 1.12E+03 & 1.13E+03 & 4.32E+01 & 3.69E+03 & 3.70E+03 & 6.54E+01 \\
    & $f_{4}$    & \textbf{2.29E+00} & \textbf{2.27E+00} & \textbf{1.61E-01} & 3.16E+00 & 3.16E+00 & 1.33E-01 & 7.41E+03 & 7.40E+03 & 8.60E+02 & 7.43E+03 & 7.42E+03 & 6.57E+02 \\
    & $f_{5}$    & \textbf{1.28E+00} & \textbf{1.28E+00} & \textbf{1.88E-02} & 1.35E+00 & 1.35E+00 & 1.44E-02 & 2.04E+01 & 2.11E+01 & 4.34E+00 & 6.40E+01 & 6.40E+01 & 5.20E-01 \\
    & $f_{6}$    & \textbf{3.30E+01} & \textbf{3.33E+01} & \textbf{1.16E+00} & 3.61E+01 & 3.62E+01 & 1.29E+00 & 2.04E+02 & 1.86E+02 & 7.12E+01 & 2.17E+02 & 1.76E+02 & 1.24E+02 \\
    & $f_{7}$    & 5.27E+03 & 5.28E+03 & 6.66E+01 & 5.35E+03 & 5.35E+03 & 7.33E+01 & \textbf{4.13E+03} & \textbf{4.14E+03} & \textbf{1.58E+02} & 7.49E+03 & 7.50E+03 & 7.49E+01 \\
    & $f_{8}$    & \textbf{1.23E+03} & \textbf{1.24E+03} & \textbf{4.30E+01} & 1.36E+03 & 1.36E+03 & 4.94E+01 & 9.32E+04 & 9.06E+04 & 1.54E+04 & 8.83E+04 & 8.57E+04 & 1.02E+04 \\
    & $f_{9}$    & \textbf{1.54E+03} & \textbf{1.54E+03} & \textbf{7.07E+01} & 1.71E+03 & 1.72E+03 & 7.53E+01 & 4.78E+03 & 4.75E+03 & 1.60E+02 & 9.85E+03 & 9.83E+03 & 3.56E+02 \\
    & $f_{10}$   & \textbf{2.54E+02} & \textbf{2.53E+02} & \textbf{6.90E+00} & 2.64E+02 & 2.62E+02 & 1.01E+01 & 9.75E+02 & 9.72E+02 & 8.45E+01 & 9.34E+02 & 9.11E+02 & 8.28E+01 \\
    & $f_{11}$   & \textbf{2.42E+03} & \textbf{2.42E+03} & \textbf{9.55E+01} & 2.53E+03 & 2.50E+03 & 1.32E+02 & \textbf{2.35E+03} & \textbf{2.32E+03} & \textbf{2.84E+02} & 5.77E+03 & 5.76E+03 & 7.31E+01 \\
    & $f_{12}$   & \textbf{5.37E+03} & \textbf{5.36E+03} & \textbf{1.27E+02} & 5.48E+03 & 5.46E+03 & 1.13E+02 & 5.47E+03 & 5.46E+03 & 2.03E+02 & 7.69E+03 & 7.69E+03 & 8.66E+01 \\
    & $f_{13}$   & \textbf{3.80E+03} & \textbf{3.81E+03} & \textbf{9.80E+01} & 3.97E+03 & 3.96E+03 & 1.13E+02 & 3.88E+03 & 3.90E+03 & 1.67E+02 & 6.38E+03 & 6.38E+03 & 7.76E+01 \\
    & $f_{14}$   & \textbf{2.47E+02} & \textbf{2.48E+02} & \textbf{6.61E+00} & 2.52E+02 & 2.52E+02 & 5.29E+00 & 9.31E+02 & 9.14E+02 & 9.48E+01 & 8.59E+02 & 8.64E+02 & 9.22E+01 \\
    & $f_{15}$   & 3.43E+03 & 3.43E+03 & 1.25E+02 & 3.57E+03 & 3.58E+03 & 1.13E+02 & \textbf{2.38E+03} & \textbf{2.35E+03} & \textbf{3.06E+02} & 5.46E+03 & 5.46E+03 & 6.56E+01 \\
\hline
\multicolumn{2}{c|}{Number of firsts} & \multicolumn{3}{c|}{13} & \multicolumn{3}{c|}{0} & \multicolumn{3}{c|}{3} & \multicolumn{3}{c}{0} \bigstrut\\
\hline
\hline
\multicolumn{14}{c}{FEs = $3 \times 10^6$}  \bigstrut\\
\hline
\multirow{2}[4]{*}{Benchmark} & \multirow{2}[4]{*}{Problem} & \multicolumn{3}{c|}{CSG} & \multicolumn{3}{c|}{GSG} & \multicolumn{3}{c|}{RDG2} & \multicolumn{3}{c}{MDG} \bigstrut\\
\cline{3-14}      &       & Mean  & Median & Std   & Mean  & Median & Std   & Mean  & Median & Std   & Mean  & Median & Std \bigstrut\\
\hline
\multirow{14}[2]{*}{BMS} & $f_{1}$ & \textbf{3.02E+01} & \textbf{3.15E+01} & \textbf{4.86E+00} & \textbf{3.02E+01} & \textbf{2.96E+01} & \textbf{7.33E+00} & 1.04E+02 & 1.05E+02 & 9.23E+00 & 1.05E+02 & 1.07E+02 & 8.34E+00 \bigstrut[t]\\
    &$f_{2}$    & \textbf{1.00E+02} & \textbf{1.00E+02} & \textbf{5.13E-14} & \textbf{1.00E+02} & \textbf{1.00E+02} & \textbf{4.78E-14} & 1.00E+02 & 1.00E+02 & 1.90E-01 & 1.00E+02 & 1.00E+02 & 1.55E-01 \\
   & $f_{3}$    & \textbf{1.65E+01} & \textbf{1.69E+01} & \textbf{3.28E+00} & \textbf{1.78E+01} & \textbf{1.79E+01} & \textbf{3.41E+00} & 9.46E+01 & 9.58E+01 & 9.34E+00 & 1.52E+03 & 1.53E+03 & 4.87E+01 \\
  &  $f_{4}$    & \textbf{8.65E-22} & \textbf{2.83E-22} & \textbf{1.64E-21} & 3.26E-21 & 9.49E-22 & 8.75E-21 & 3.34E+02 & 2.70E+02 & 1.98E+02 & 2.53E+02 & 2.37E+02 & 7.63E+01 \\
  &  $f_{5}$    & \textbf{1.00E+00} & \textbf{1.00E+00} & \textbf{1.07E-15} & \textbf{1.00E+00} & \textbf{1.00E+00} & \textbf{1.02E-15} & 1.00E+00 & 1.00E+00 & 8.38E-06 & 4.29E+01 & 4.31E+01 & 1.27E+00 \\
 &  $f_{6}$    & \textbf{1.13E+00} & \textbf{1.12E+00} & \textbf{1.29E-02} & \textbf{1.13E+00} & \textbf{1.13E+00} & \textbf{1.65E-02} & 4.84E+00 & 4.47E+00 & 8.94E-01 & 3.89E+00 & 3.90E+00 & 2.49E-01 \\
  &  $f_{7}$    & \textbf{1.41E+02} & \textbf{1.43E+02} & \textbf{7.17E+00} & \textbf{1.40E+02} & \textbf{1.42E+02} & \textbf{6.66E+00} & 1.92E+02 & 1.87E+02 & 2.95E+01 & 2.42E+03 & 2.46E+03 & 3.12E+02 \\
 &  $f_{8}$    & \textbf{1.09E+02} & \textbf{1.09E+02} & \textbf{1.10E+00} & \textbf{1.09E+02} & \textbf{1.08E+02} & \textbf{9.89E-01} & 2.56E+02 & 2.53E+02 & 1.86E+01 & 2.54E+02 & 2.53E+02 & 1.87E+01 \\
 &   $f_{9}$   & \textbf{1.04E+02} & \textbf{1.04E+02} & \textbf{7.69E-01} & \textbf{1.04E+02} & \textbf{1.04E+02} & \textbf{4.36E-01} & 4.12E+02 & 2.87E+02 & 3.02E+02 & 4.64E+03 & 4.65E+03 & 1.12E+02 \\
    & $f_{10}$   & \textbf{1.03E+02} & \textbf{1.03E+02} & \textbf{4.89E-01} & \textbf{1.03E+02} & \textbf{1.03E+02} & \textbf{4.95E-01} & 1.42E+02 & 1.42E+02 & 5.19E+00 & 1.40E+02 & 1.42E+02 & 3.63E+00 \\
    & $f_{11}$   & \textbf{1.33E+02} & \textbf{1.34E+02} & \textbf{4.32E+00} & \textbf{1.35E+02} & \textbf{1.35E+02} & \textbf{3.75E+00} & 1.54E+02 & 1.56E+02 & 7.81E+00 & 7.54E+02 & 7.01E+02 & 1.69E+02 \\
    & $f_{12}$   & 1.59E+03 & 1.59E+03 & 4.34E+01 & 1.59E+03 & 1.59E+03 & 3.13E+01 & \textbf{1.45E+03} & \textbf{1.45E+03} & \textbf{4.84E+01} & 2.33E+03 & 2.19E+03 & 3.32E+02 \\
   & $f_{13}$   & \textbf{1.45E+02} & \textbf{1.47E+02} & \textbf{6.25E+00} & \textbf{1.48E+02} & \textbf{1.49E+02} & \textbf{5.52E+00} & 1.52E+02 & 1.50E+02 & 7.32E+00 & 8.76E+02 & 7.54E+02 & 2.73E+02 \\
    & $f_{14}$   & \textbf{5.65E+01} & \textbf{5.62E+01} & \textbf{3.04E+00} & \textbf{5.80E+01} & \textbf{5.83E+01} & \textbf{3.08E+00} & 1.72E+02 & 1.70E+02 & 8.52E+00 & 1.70E+02 & 1.69E+02 & 7.49E+00 \\
    & $f_{15}$   & \textbf{1.54E+02} & \textbf{1.54E+02} & \textbf{5.31E+00} & \textbf{1.54E+02} & \textbf{1.54E+02} & \textbf{6.08E+00} & 2.02E+02 & 2.02E+02 & 5.95E+00 & 6.01E+02 & 5.47E+02 & 1.92E+02 \\
\hline
\multicolumn{2}{c|}{Number of firsts} & \multicolumn{3}{c|}{14} & \multicolumn{3}{c|}{13} & \multicolumn{3}{c|}{1} & \multicolumn{3}{c}{0} \bigstrut\\
\hline
\end{tabular}%
    }
  \label{tab OptimizationResults1}%
  \vspace{-0.44cm}
\end{table*}%

\begin{figure*}[htbp]
\centering
 \captionsetup{labelfont={color=black}}
\subfigcapskip=-6pt
\setlength{\abovecaptionskip}{-2pt}
 \subfigure[{$f_9$ with $1.2\times{10}^5$ FEs}]{
 \includegraphics[width=4.5cm]{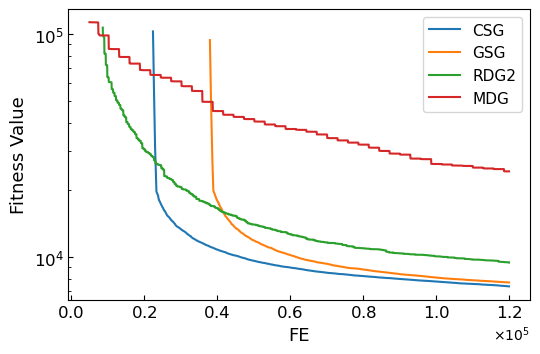}
 }
  \subfigure[{$f_9$ with $6\times{10}^5$ FEs}]{
 \includegraphics[width=4.5cm]{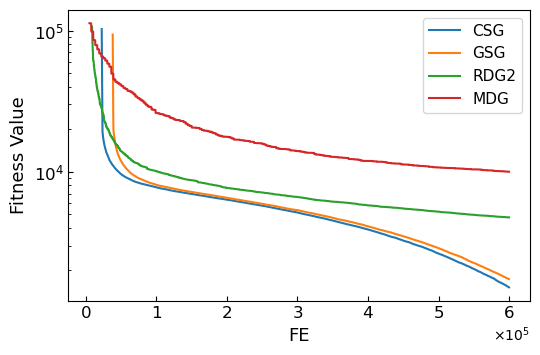}
 }
  \subfigure[{$f_9$ with $3\times{10}^6$ FEs}]{
 \includegraphics[width=4.5cm]{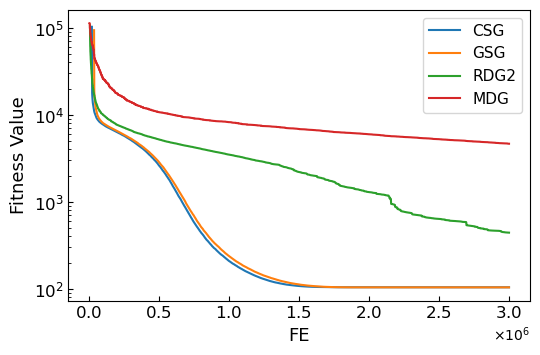}
 }
 \vspace{-0.3cm}
 
 \subfigure[{$f_{11}$ with $1.2\times{10}^5$ FEs}]{
 \includegraphics[width=4.5cm]{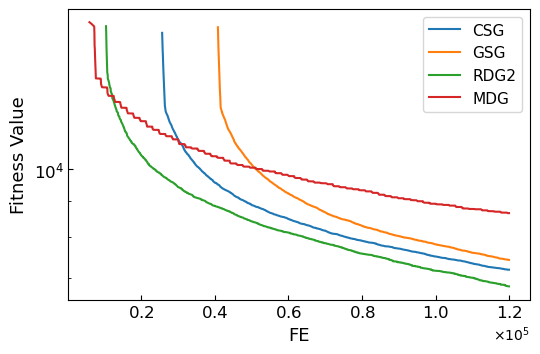}
 }
  \subfigure[{$f_{11}$ with $6\times{10}^5$ FEs}]{
 \includegraphics[width=4.5cm]{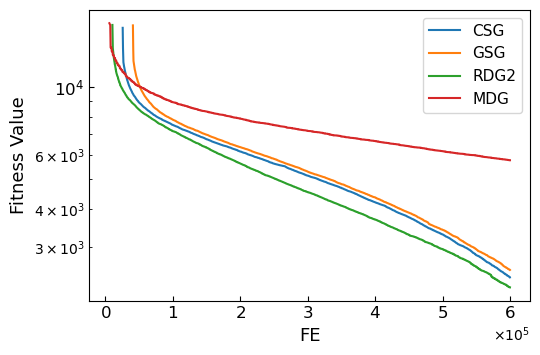}
 }
   \subfigure[{$f_{11}$ with $3\times{10}^6$ FEs}]{
 \includegraphics[width=4.5cm]{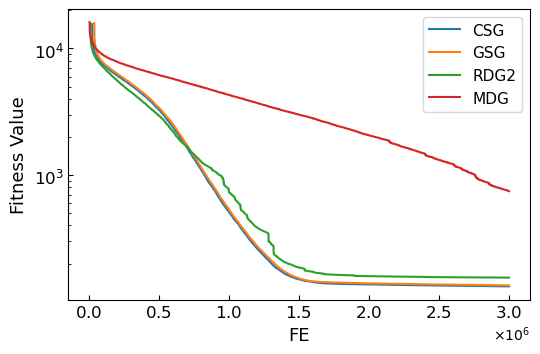}
 }
 \caption{{Convergence curves for the CSG, GSG, RDG2, and MDG methods when embedded into the DECC framework to solve the 1000-D $f_9$ and $f_{11}$ in BMS benchmark with the different termination conditions. The horizontal axis represents the number of FEs used in the decomposition and optimization process. The vertical axis represents the average fitness value.}}
 \label{fcc_mode_classification}
 \vspace{-0.5cm}
\end{figure*}

\begin{table*}[t]
  \centering
  \caption{The optimization results of various grouping methods on 5000-D BMS problems over 10 independent runs. Three subtables record the optimization results for fes of $6 \times 10^5$, $3 \times 10^6$ and $1.5 \times 10^7$, respectively.}
  \resizebox{\linewidth}{!}{
       \renewcommand{\arraystretch}{01}
 \renewcommand{\tabcolsep}{10pt}
\begin{tabular}{cc|ccc|ccc|ccc|ccc}
\hline
\multicolumn{14}{c}{FEs = $6$ $\times$ $10^5$}\bigstrut\\
\hline
\multirow{2}[4]{*}{Benchmark} & \multirow{2}[4]{*}{Problem} & \multicolumn{3}{c|}{CSG} & \multicolumn{3}{c|}{GSG} & \multicolumn{3}{c|}{RDG2} & \multicolumn{3}{c}{MDG} \bigstrut\\
\cline{3-14}      &       & Mean  & Median & Std   & Mean  & Median & Std   & Mean  & Median & Std   & Mean  & Median & Std \bigstrut\\
\hline
\multirow{14}[2]{*}{BMS} & $f_{1}$ & \textbf{5.98E+04} & \textbf{5.98E+04} & \textbf{5.57E+02} & 7.30E+04 & 7.31E+04 & 8.35E+02 & 5.92E+06 & 6.00E+06 & 4.93E+05 & 6.09E+06 & 6.07E+06 & 3.82E+05 \bigstrut[t]\\
    & $f_{2}$    & \textbf{4.29E+02} & \textbf{4.28E+02} & \textbf{5.94E+00} & 9.71E+02 & 9.69E+02 & 1.91E+01 & 3.98E+03 & 3.92E+03 & 2.31E+02 & 3.88E+03 & 3.86E+03 & 3.31E+02 \\
    & $f_{3}$    & \textbf{9.98E+04} & \textbf{9.98E+04} & \textbf{1.21E+02} & 1.01E+05 & 1.01E+05 & 1.91E+02 & 1.02E+05 & 1.02E+05 & 1.37E+02 & 1.14E+05 & 1.14E+05 & 2.97E+02 \\
    & $f_{4}$    & \textbf{1.63E+04} & \textbf{1.64E+04} & \textbf{2.12E+02} & 2.10E+04 & 2.09E+04 & 2.77E+02 & 2.42E+05 & 2.43E+05 & 4.54E+03 & 2.38E+05 & 2.37E+05 & 4.62E+03 \\
    & $f_{5}$    & \textbf{7.17E+01} & \textbf{7.17E+01} & \textbf{1.06E-01} & 7.36E+01 & 7.36E+01 & 1.57E-01 & 7.93E+01 & 7.93E+01 & 3.19E-01 & 2.52E+02 & 2.50E+02 & 2.75E+01 \\
    & $f_{6}$    & \textbf{9.23E+04} & \textbf{9.23E+04} & \textbf{9.42E+02} & 1.19E+05 & 1.19E+05 & 1.33E+03 & 4.70E+05 & 4.68E+05 & 2.98E+04 & 4.98E+05 & 5.00E+05 & 4.83E+04 \\
    & $f_{7}$    & \textbf{2.10E+04} & \textbf{2.10E+04} & \textbf{1.84E+02} & 2.23E+04 & 2.23E+04 & 1.29E+02 & 2.60E+04 & 2.60E+04 & 1.88E+02 & 4.33E+04 & 4.36E+04 & 1.38E+03 \\
    & $f_{8}$    & \textbf{1.28E+04} & \textbf{1.28E+04} & \textbf{1.99E+02} & 1.85E+04 & 1.85E+04 & 3.11E+02 & 1.97E+05 & 1.96E+05 & 6.02E+03 & 1.90E+05 & 1.91E+05 & 8.14E+03 \\
    & $f_{9}$    & \textbf{9.45E+03} & \textbf{9.44E+03} & \textbf{1.55E+01} & 9.78E+03 & 9.79E+03 & 2.33E+01 & 2.06E+04 & 2.08E+04 & 5.79E+02 & 4.36E+04 & 4.38E+04 & 2.75E+03 \\
    & $f_{10}$   & 2.00E+05 & 1.99E+05 & 1.07E+04 & \textbf{1.49E+05} & \textbf{1.47E+05} & \textbf{7.59E+03} & 1.99E+05 & 2.00E+05 & 5.94E+03 & 1.81E+05 & 1.82E+05 & 4.30E+03 \\
    & $f_{11}$   & \textbf{1.73E+04} & \textbf{1.73E+04} & \textbf{2.51E+02} & 1.80E+04 & 1.80E+04 & 3.37E+02 & 1.83E+04 & 1.84E+04 & 1.67E+02 & 3.07E+04 & 3.10E+04 & 5.16E+02 \\
    & $f_{12}$   & \textbf{2.95E+04} & \textbf{2.95E+04} & \textbf{1.41E+02} & 3.07E+04 & 3.08E+04 & 1.59E+02 & 3.23E+04 & 3.23E+04 & 2.03E+02 & 4.35E+04 & 4.34E+04 & 4.27E+02 \\
    & $f_{13}$   & \textbf{1.72E+04} & \textbf{1.72E+04} & \textbf{9.44E+01} & 1.80E+04 & 1.80E+04 & 7.64E+01 & 2.15E+04 & 2.15E+04 & 1.77E+02 & 3.03E+04 & 3.05E+04 & 4.00E+02 \\
    & $f_{14}$   & \textbf{1.25E+04} & \textbf{1.24E+04} & \textbf{9.06E+01} & 1.40E+04 & 1.40E+04 & 1.87E+02 & 3.19E+04 & 3.18E+04 & 6.54E+02 & 3.12E+04 & 3.11E+04 & 4.78E+02 \\
    & $f_{15}$   & \textbf{2.54E+04} & \textbf{2.54E+04} & \textbf{1.27E+02} & 2.62E+04 & 2.62E+04 & 1.00E+02 & 2.74E+04 & 2.74E+04 & 1.81E+02 & 3.97E+04 & 3.94E+04 & 7.40E+02 \\
\hline
\multicolumn{2}{c|}{Number of firsts} & \multicolumn{3}{c|}{14} & \multicolumn{3}{c|}{1} & \multicolumn{3}{c|}{0} & \multicolumn{3}{c}{0} \bigstrut\\
\hline
\hline
\multicolumn{14}{c}{FEs = $3 \times 10^6$} \bigstrut\\
\hline
\multirow{2}[4]{*}{Benchmark} & \multirow{2}[4]{*}{Problem} & \multicolumn{3}{c|}{CSG} & \multicolumn{3}{c|}{GSG} & \multicolumn{3}{c|}{RDG2} & \multicolumn{3}{c}{MDG} \bigstrut\\
\cline{3-14}      &       & Mean  & Median & Std   & Mean  & Median & Std   & Mean  & Median & Std   & Mean  & Median & Std \bigstrut\\
\hline
\multirow{14}[2]{*}{BMS} & $f_{1}$& \textbf{1.98E+04} & \textbf{1.98E+04} & \textbf{2.14E+02} & 2.08E+04 & 2.08E+04 & 2.02E+02 & 1.71E+06 & 1.71E+06 & 1.61E+05 & 1.76E+06 & 1.74E+06 & 1.06E+05  \bigstrut[t]\\
    & $f_{2}$    & \textbf{1.00E+02} & \textbf{1.00E+02} & \textbf{5.27E-05} & 1.00E+02 & 1.00E+02 & 1.43E-04 & 8.88E+02 & 8.88E+02 & 4.13E+01 & 8.94E+02 & 8.76E+02 & 5.36E+01 \\
    & $f_{3}$    & \textbf{7.20E+04} & \textbf{7.20E+04} & \textbf{1.98E+02} & 7.26E+04 & 7.28E+04 & 2.90E+02 & 8.74E+04 & 8.74E+04 & 1.28E+02 & 1.00E+05 & 1.00E+05 & 1.60E+02 \\
    & $f_{4}$    & \textbf{6.92E+00} & \textbf{6.87E+00} & \textbf{2.09E-01} & 8.23E+00 & 8.27E+00 & 1.78E-01 & 1.20E+05 & 1.20E+05 & 5.28E+03 & 1.19E+05 & 1.19E+05 & 3.29E+03 \\
    & $f_{5}$    & \textbf{2.02E+00} & \textbf{2.03E+00} & \textbf{2.31E-02} & 2.17E+00 & 2.16E+00 & 2.54E-02 & 6.58E+01 & 6.57E+01 & 5.43E-01 & 1.21E+02 & 1.21E+02 & 1.81E+00 \\
    & $f_{6}$    & \textbf{8.48E+01} & \textbf{8.52E+01} & \textbf{1.46E+00} & 9.23E+01 & 9.21E+01 & 1.25E+00 & 1.14E+05 & 1.14E+05 & 1.64E+04 & 1.48E+05 & 1.18E+05 & 7.44E+04 \\
    & $f_{7}$    & \textbf{1.23E+04} & \textbf{1.23E+04} & \textbf{7.23E+01} & 1.24E+04 & 1.24E+04 & 3.61E+01 & 1.27E+04 & 1.27E+04 & 8.44E+01 & 2.43E+04 & 2.43E+04 & 3.48E+02 \\
    & $f_{8}$    & \textbf{6.03E+02} & \textbf{6.04E+02} & \textbf{9.07E+00} & 6.46E+02 & 6.45E+02 & 9.15E+00 & 6.09E+04 & 6.00E+04 & 4.00E+03 & 5.91E+04 & 6.02E+04 & 3.59E+03 \\
    & $f_{9}$    & \textbf{2.99E+03} & \textbf{2.99E+03} & \textbf{4.75E+01} & 3.19E+03 & 3.20E+03 & 5.39E+01 & 1.19E+04 & 1.18E+04 & 3.22E+02 & 2.12E+04 & 2.12E+04 & 7.10E+02 \\
    & $f_{10}$   & 5.35E+04 & 5.31E+04 & 3.17E+03 & \textbf{1.23E+04} & \textbf{1.23E+04} & \textbf{1.15E+03} & 4.93E+04 & 4.97E+04 & 2.10E+03 & 4.82E+04 & 4.85E+04 & 2.18E+03 \\
    & $f_{11}$   & 1.04E+04 & 1.04E+04 & 9.92E+01 & 1.06E+04 & 1.06E+04 & 9.97E+01 & \textbf{9.85E+03} & \textbf{9.84E+03} & \textbf{3.84E+01} & 1.84E+04 & 1.84E+04 & 2.28E+02 \\
    & $f_{12}$   & \textbf{2.02E+04} & \textbf{2.02E+04} & \textbf{9.07E+01} & 2.03E+04 & 2.03E+04 & 1.01E+02 & 2.05E+04 & 2.06E+04 & 1.52E+02 & 2.85E+04 & 2.85E+04 & 1.42E+02 \\
    & $f_{13}$   & \textbf{1.06E+04} & \textbf{1.07E+04} & \textbf{2.06E+01} & 1.07E+04 & 1.07E+04 & 4.43E+01 & 1.17E+04 & 1.17E+04 & 6.07E+01 & 1.81E+04 & 1.81E+04 & 1.48E+02 \\
    & $f_{14}$   & \textbf{5.20E+03} & \textbf{5.22E+03} & \textbf{8.68E+01} & 5.28E+03 & 5.25E+03 & 9.81E+01 & 1.41E+04 & 1.40E+04 & 4.14E+02 & 1.43E+04 & 1.43E+04 & 2.35E+02 \\
    & $f_{15}$   & \textbf{1.82E+04} & \textbf{1.82E+04} & \textbf{3.95E+01} & 1.83E+04 & 1.83E+04 & 4.29E+01 & 1.85E+04 & 1.85E+04 & 6.21E+01 & 2.75E+04 & 2.75E+04 & 1.39E+02 \\

\hline
\multicolumn{2}{c|}{Number of firsts} & \multicolumn{3}{c|}{13} & \multicolumn{3}{c|}{1} & \multicolumn{3}{c|}{1} & \multicolumn{3}{c}{0} \bigstrut\\
\hline
\hline
\multicolumn{14}{c}{FEs = $1.5 \times 10^7$}  \bigstrut\\
\hline
\multirow{2}[4]{*}{Benchmark} & \multirow{2}[4]{*}{Problem} & \multicolumn{3}{c|}{CSG} & \multicolumn{3}{c|}{GSG} & \multicolumn{3}{c|}{RDG2} & \multicolumn{3}{c}{MDG} \bigstrut\\
\cline{3-14}      &       & Mean  & Median & Std   & Mean  & Median & Std   & Mean  & Median & Std   & Mean  & Median & Std \bigstrut\\
\hline
\multirow{14}[2]{*}{BMS} & $f_{1}$ & \textbf{2.28E+03} & \textbf{2.28E+03} & \textbf{4.91E+01} & \textbf{2.30E+03} & \textbf{2.30E+03} & \textbf{4.67E+01} & 8.32E+04 & 8.28E+04 & 9.44E+03 & 8.74E+04 & 8.93E+04 & 6.84E+03 \bigstrut[t]\\
    & $f_{2}$    & \textbf{1.00E+02} & \textbf{1.00E+02} & \textbf{0.00E+00} & \textbf{1.00E+02} & \textbf{1.00E+02} & \textbf{0.00E+00} & 2.42E+02 & 2.38E+02 & 1.23E+01 & 2.46E+02 & 2.43E+02 & 7.40E+00 \\
    & $f_{3}$    & \textbf{5.20E+02} & \textbf{5.20E+02} & \textbf{2.02E+01} & \textbf{5.15E+02} & \textbf{5.10E+02} & \textbf{1.90E+01} & 7.71E+04 & 7.70E+04 & 3.43E+02 & 8.51E+04 & 8.51E+04 & 2.07E+02 \\
    & $f_{4}$    & \textbf{8.51E-21} & \textbf{5.53E-21} & \textbf{8.89E-21} & \textbf{1.02E-20} & \textbf{9.31E-21} & \textbf{7.00E-21} & 4.55E+04 & 4.53E+04 & 3.36E+03 & 4.58E+04 & 4.59E+04 & 3.30E+03 \\
    & $f_{5}$    & \textbf{1.00E+00} & \textbf{1.00E+00} & \textbf{0.00E+00} & \textbf{1.00E+00} & \textbf{1.00E+00} & \textbf{3.23E-15} & 3.34E+01 & 3.37E+01 & 3.08E+00 & 8.72E+01 & 8.71E+01 & 2.48E-01 \\
    & $f_{6}$    & \textbf{1.17E+00} & \textbf{1.17E+00} & \textbf{1.04E-02} & \textbf{1.17E+00} & \textbf{1.17E+00} & \textbf{8.16E-03} & 5.94E+03 & 4.86E+03 & 2.74E+03 & 1.23E+04 & 2.92E+03 & 2.32E+04 \\
    & $f_{7}$    & \textbf{1.09E+03} & \textbf{1.09E+03} & \textbf{1.03E+01} & \textbf{1.09E+03} & \textbf{1.09E+03} & \textbf{1.27E+01} & 7.95E+03 & 7.94E+03 & 7.38E+01 & 1.18E+04 & 1.17E+04 & 8.48E+01 \\
    & $f_{8}$    & \textbf{1.04E+02} & \textbf{1.04E+02} & \textbf{3.51E-01} & \textbf{1.04E+02} & \textbf{1.04E+02} & \textbf{2.63E-01} & 1.06E+04 & 1.05E+04 & 1.24E+03 & 5.91E+04 & 6.02E+04 & 3.59E+03 \\
    & $f_{9}$    & \textbf{1.05E+02} & \textbf{1.05E+02} & \textbf{2.31E-01} & \textbf{1.05E+02} & \textbf{1.05E+02} & \textbf{3.36E-01} & 7.60E+03 & 7.62E+03 & 6.34E+01 & 1.24E+04 & 1.23E+04 & 4.05E+02 \\
    & $f_{10}$   & 9.39E+03 & 9.49E+03 & 7.76E+02 & \textbf{6.83E+02} & \textbf{6.69E+02} & \textbf{7.10E+01} & 6.29E+03 & 5.84E+03 & 1.35E+03 & 7.01E+03 & 6.57E+03 & 1.34E+03 \\
    & $f_{11}$   & \textbf{1.36E+03} & \textbf{1.35E+03} & \textbf{3.49E+01} & \textbf{1.40E+03} & \textbf{1.40E+03} & \textbf{2.79E+01} & 6.86E+03 & 6.88E+03 & 7.23E+01 & 9.70E+03 & 9.70E+03 & 4.98E+01 \\
    & $f_{12}$   & \textbf{8.40E+03} & \textbf{8.41E+03} & \textbf{9.18E+01} & \textbf{8.36E+03} & \textbf{8.40E+03} & \textbf{1.18E+02} & 1.37E+04 & 1.37E+04 & 1.48E+02 & 1.79E+04 & 1.79E+04 & 1.16E+02 \\
    & $f_{13}$   & \textbf{1.06E+03} & \textbf{1.06E+03} & \textbf{5.29E+00} & \textbf{1.06E+03} & \textbf{1.05E+03} & \textbf{7.81E+00} & 7.44E+03 & 7.45E+03 & 4.61E+01 & 9.70E+03 & 9.71E+03 & 4.12E+01 \\
    & $f_{14}$   & \textbf{1.24E+03} & \textbf{1.25E+03} & \textbf{3.41E+01} & \textbf{1.24E+03} & \textbf{1.23E+03} & \textbf{3.77E+01} & 7.14E+03 & 7.10E+03 & 2.96E+02 & 7.14E+03 & 7.15E+03 & 1.78E+02 \\
    & $f_{15}$   & \textbf{5.94E+03} & \textbf{5.94E+03} & \textbf{8.16E+01} & \textbf{5.97E+03} & \textbf{5.97E+03} & \textbf{7.20E+01} & 1.49E+04 & 1.50E+04 & 1.07E+02 & 1.85E+04 & 1.85E+04 & 5.68E+01 \\

\hline
\multicolumn{2}{c|}{Number of firsts} & \multicolumn{3}{c|}{14} & \multicolumn{3}{c|}{15} & \multicolumn{3}{c|}{0} & \multicolumn{3}{c}{0} \bigstrut\\
\hline
\end{tabular}%
    }
  \label{tab OptimizationResults}%
  \vspace{-0.5cm}
\end{table*}%

\subsection{Comparison on Optimization}
Based on the grouping results obtained from CSG, GSG, RDG2, and MDG, optimization experiments are conducted on BMS.
The grouping outcomes are integrated into the CC framework for optimization, utilizing the self-adaptive differential evolution with neighborhood search (SaNSDE) \cite{SANSDE} as the selected optimizer.
Fully separable variables are grouped into subgroups of size 50 for optimization.
As per \cite{SANSDE}, the population size of the algorithms is set at 50.

\subsubsection{Optimization on 1000-dimensional Problems}
The optimization results are obtained through 25 independent runs.
Wilcoxon's rank sum test at the 0.05 significance level is performed on the optimization results, and the results with a significant advantage are highlighted in bold.
The number of FEs for each problem is set to $3 \times 10^6$ as the termination criterion for each optimization process.
The optimization results of each algorithm are recorded when FEs reach $1.2 \times 10^5$ and $6 \times 10^5$ to compare the performance of each algorithm during the optimization process \cite{CEC2010}.
The optimization results of the respective algorithms on 1000-dimensional problems are presented in Table \ref{tab OptimizationResults1}. 
The convergence curves for $f_9$ and $f_{11}$ are presented in Fig. 2.
Due to the page limit, the remaining convergence curves are provided in Section S.I of the Supplementary Material.
It can be observed that in the early optimization stage, CSG exhibits better optimization results than GSG does. 
Finally, CSG and GSG achieve the best similar optimization results due to the accurate decomposition.

At the beginning of the optimization process, when FEs is $1.2 \times 10^5$, CSG achieves 9 first-place results, while RDG2 achieves 6 first-place results. 
It is evident that CSG outperforms GSG in terms of optimization results. 
This can be attributed to CSG's consumption of fewer computational resources during the grouping stage and its earlier execution of the optimization stage. 

As the optimization stage progresses, the availability of computing resources increases, and the advantage of reduced computational resource consumption during the grouping stage of RDG2 gradually diminishes. 
When FEs reach $6 \times 10^5$, CSG achieves 13 first-place results, while RDG2 achieves 3 first-place results.

Towards the end of the optimization stage, when computational resources are sufficient, CSG and GSG produce comparable results due to similar grouping outcomes.
Thanks to the accurate grouping achieved by CSG and GSG, both CSG and GSG exhibit superior optimization results compared to RDG2 and MDG when FEs is $3 \times 10^6$, except for problem $f_{12}$.
The better performance of RDG2 on $f_{12}$ may be attributed to its accurate grouping of all non-separable variables.
Although RDG2 incorrectly identifies multiple separable variables as one non-separable variable group, this has a lesser negative impact on the optimization process.
RDG2 significantly outperforms MDG for certain problems such as $f_5$, $f_7$ and $f_{13}$.
MDG, based on a binary-tree-based iterative merging process, tends to group fully separable variables into numerous small subgroups, thereby affecting the optimization results of the MDG algorithm.

\subsubsection{Optimization on 5000-dimensional Problems}
The optimization experimental results on 5000-D problems are presented in Table \ref{tab OptimizationResults}.
{Due to the page limit, the remaining convergence curves are presented in Section S.I of the Supplementary Material.}
The number of FEs is set to $1.5 \times 10^7$  as the termination criterion for each optimization process, and the optimization results of each algorithm are recorded at FEs of $6 \times 10^5$ and $3 \times 10^6$.
The 5000-dimensional optimization results are obtained from 10 independent runs \cite{CSO}.

The results indicate that DECC with CSG significantly outperforms GSG, RDG2, and MDG in the early optimization stage when FEs is $6 \times 10^5$.
CSG achieves 14 first-place results in the initial optimization stage. 
This can be attributed to CSG's more accurate decomposition of all problems with lower computational resource consumption.

As the optimization progresses, the impact of computational resources on the grouping stage gradually diminishes, and GSG gradually achieves similar results to CSG. 
When FEs reach $1.5 \times 10^7$, CSG obtains 14 first-place results, while GSG obtains 15 first-place results. 
However, for $f_{10}$, GSG incorrectly identifies a large non-separable variable group as multiple smaller groups of non-separable variables, yet it achieves better optimization results compared to the accurate grouping of CSG. 
One possible explanation for this outcome is that the subgroup size in the CC framework of CSG is excessively large, negatively impacting the optimization process.

Overall, CSG significantly enhances the optimization efficiency of GSG in the early stages of optimization. 
Towards the end of the optimization process, the precise grouping achieved by both CSG and GSG leads to improved optimization results.

\subsubsection{Optimization on Classical Problems}
Due to the page limit, the detailed optimization results for CEC2010, CEC2013 and BNS benchmarks are shown in Section S.III of the Supplementary Material. 

Overall, CSG outperforms the compared algorithms in terms of optimization results across the three benchmarks.
In comparison to GSG, CSG utilizes fewer computational resources in the grouping phase, allowing it to enter the optimization stage earlier with more available computational resources. 
Consequently, CSG exhibits superior performance to GSG throughout the optimization process.
RDG2 and MDG achieve good optimization performance in dealing with additively separable problems in the CEC benchmark. 
However, both RDG2 and MDG are less effective in dealing with non-additively separable problems due to their inability to decompose such problems effectively.

\section{Conclusion}\label{section Conclusion}
In this article, we conduct a thorough analysis of the strengths and weaknesses of DG series methods and GSG algorithm, and we find that they complement each other. 
Therefore, we propose a composite algorithm as a novel approach to decompose LSGO problems, integrating the advantages of both methods while addressing their limitations. 
Based on these insights, we introduce a composite separability grouping (CSG) method that achieves more accurate problem decomposition with lower computational complexity.

In the proposed CSG, we utilize a one-to-all detection approach to sequentially identify additively separable variables, multiplicatively separable variables, and generally separable variables. 
Once all separable variables are detected, non-separable variables are grouped by analyzing their interactions with the formed groups of non-separable variables. 
We introduce the MSVD method, which effectively detects the more prevalent multiplicatively separable variables in the problem by leveraging historical information from previous stages to conserve computational resources. 
Additionally, we propose the NVG method to efficiently group non-separable variables through interaction detection with existing variable groups.

We then introduce a new LSGO benchmark, called BMS, which includes multiple types of separable variables within each benchmark function. 
BMS combines the fundamental functions used in BNS \cite{GSG} and CEC \cite{CEC2013}, addressing the limitation of previous benchmark functions that only contained one type of separable variables.

{In the experiments, we compare CSG with GSG, RDG2, and MDG on BMS, CEC and BNS benchmarks.}
The experimental results demonstrate that CSG significantly reduces computational resource consumption compared to GSG. 
Moreover, CSG outperforms RDG2 and MDG in terms of decomposition accuracy and optimization results. 
We also showcase the high efficiency of NVG in grouping non-separable variables, as well as a more efficient approach to handling separable variables in the optimization process.

While CSG successfully mitigates the issue of excessive computational resource consumption in GSG, there are still avenues for further exploration. 
The definition of separability can be expanded to guide problem decomposition more effectively. 
Additionally, we can explore more generalized mathematical optimization-based decomposition for LSGO problems. 
Lastly, leveraging the computational power of deep learning to explore problem structures represents a potential research direction.

\footnotesize

\bibliography{TAI_template}

\end{document}